\documentclass[12pt,oneside,a4paper,reqno]{amsart}

\usepackage{amssymb}
\usepackage{amsmath}
\usepackage{amsthm}
\usepackage{amscd}
\usepackage[all]{xy}
\usepackage{longtable}
\usepackage{mathrsfs}
\usepackage[dvipdfmx]{xcolor}
\usepackage[dvipdfmx]{pict2e}
\usepackage[dvipdfmx]{graphicx}

\usepackage{comment}
\usepackage{todonotes}

\usepackage[dvipdfmx]{hyperref}
\usepackage{hyperref}
\hypersetup{
	colorlinks=true,
	linkcolor=red,
	citecolor=blue}
	\usepackage[utf8]{inputenc}
\usepackage[T1]{fontenc}
\usepackage{color}

\usepackage{extarrows}

\usepackage{tikz}
\usepackage{tikz-cd}
\usetikzlibrary{arrows.meta}

\usepackage{amscd}

\usepackage{geometry}
\theoremstyle{plain}
\newtheorem{theorem}{Theorem}[section]
\newtheorem{lemma}[theorem]{Lemma}
\newtheorem{corollary}[theorem]{Corollary}
\newtheorem{proposition}[theorem]{Proposition}
\newtheorem{remark}[theorem]{Remark}

\theoremstyle{definition}
\newtheorem{definition}[theorem]{Definition}

\newtheorem{conjecture}[theorem]{Conjecture}

\newtheorem{remark-theorem}[theorem]{Remark-Theorem}

\newcommand{\kah}{K\"{a}hler }
\newcommand{\idd}{i\partial\overline{\partial}}
\newcommand{\dbar}{\overline{\partial}}

\newcommand{\cal}[1]{\mathcal{#1}}
\newcommand{\bb}[1]{\mathbb{#1}}
\newcommand{\scr}[1]{\mathscr{#1}}
\newcommand{\rom}[1]{\mathrm{#1}}

\newcommand{\xs}{X_{sing}}
\newcommand{\reg}{X_{reg}}
\newcommand{\ogr}{\omega_X^{GR}}
\newcommand{\tx}{\widetilde{X}}
\newcommand{\dvo}{dV_\omega}

\newcommand{\tl}[1]{\widetilde{#1}}

\newcommand{\exc}{\rom{Exc}}

\newcommand{\ureg}{U_{reg}}

\newcommand{\lara}[2]{\langle{#1},{#2}\rangle}

\newcommand{\iO}[1]{i\Theta_{#1}}

\newcommand{\oid}[1]{\otimes\mathrm{id}_{#1}}


\begin{document}
\title
[$L^2$-Dolbeault resolutions and Nadel vanishing] 
{$L^2$-Dolbeault resolutions and Nadel vanishing \\ on weakly pseudoconvex complex spaces \\ with singular Hermitian metrics}
\author{Yuta Watanabe}
\address{Department of Mathematics, Faculty of Science and Engineering, Chuo University.
1-13-27 Kasuga, Bunkyo-ku, Tokyo 112-8551, Japan}
\email{{\tt wyuta.math@gmail.com}, {\tt wyuta@math.chuo-u.ac.jp}}

\begin{abstract}
    In this paper, in order to develop a more general $L^2$-theory for the $\overline{\partial}$-operator on complex spaces, we provide $L^2$-Dolbeault fine resolutions and isomorphisms, and $L^2$-estimates, 
    for holomorphic line bundles on complex spaces equipped with singular Hermitian metrics. As applications, we obtain several generalizations of the Nadel vanishing theorem. 
\end{abstract}


\maketitle

\vspace{-5mm}

\setcounter{tocdepth}{1}
\tableofcontents

\vspace{-7mm}

\section{Introduction}

The $L^2$-theory for the $\dbar$-operator has become a fundamental and indispensable subject in complex analysis through the pioneering works on $L^2$-estimates and existence theorems (see, e.g., \cite{Hor65,Hor90,AV65}). 
Subsequently, complex geometry has been further developed through applications combining the Ohsawa-Takegoshi $L^2$-extension theorem (see \cite{OT87}) with singular Hermitian metrics on holomorphic vector bundles and their positivity (see \cite{Dem82,Dem90,BP08,Ber09,PT18,HPS18,DNWZ23,Li24,Wat26,IMW26}). 
As a consequence, singular Hermitian metrics and multiplier ideal sheaves have come to play an increasingly important role in complex geometry.
On the other hand, while the $L^2$-theory for the $\dbar$-operator has been highly developed on complex manifolds, establishing a satisfactory $L^2$-theory on complex spaces still involves substantial difficulties. 
Among the central issues are $L^2$-Dolbeault fine resolutions and isomorphism theorems, together with $L^2$-estimates.
Nevertheless, partial progress has been made in the case of smooth Hermitian metrics \cite{Rup14} and in the case where the singularities of the Hermitian metric are of line bundle type and an $\bb{R}$-polarized variation of Hodge structure (hereafter, an $\bb{R}$-PVHS) exists \cite{SZ23,SZ24}.

In this paper, we study these fundamental problems for holomorphic line bundles equipped with singular Hermitian metrics and further twisted by holomorphic vector bundles on complex spaces. 
Our approach is based on the strong openness property, $L^2$-estimates arising from Demailly's approximation theorem, and related techniques. As an application, we also prove the corresponding Nadel-type vanishing theorems.

Let $X$ be a complex space of pure dimension $n$. 
The \textit{Grauert}-\textit{Riemenschneider canonical sheaf} $\omega_X^{GR}$ on $X$ is defined by $\Gamma(U,\omega_X^{GR})=\{s\in\Gamma(U\cap\reg,\omega_{\reg})\mid i^{n^2}s\wedge\overline{s}\in L^1_{loc}(U)\}$
for any open subset $U\subset X$. 
For each locally integrable function $\varphi\in L^1_{loc}(X)$, we denote by $\ogr(\varphi)$ the $\cal{O}_X$-sheaf on $X$ defined by 
\begin{align*}
    \Gamma(U,\omega_X^{GR}(\varphi))=\{s\in\Gamma(U\cap\reg,\omega_{\reg})\mid i^{n^2}s\wedge\overline{s}\cdot e^{-2\varphi}\in L^1_{loc}(U)\}
\end{align*}
for any open subset $U\subset X$ (see \cite[Definition 2.3]{Li24}). In particular, for a singular Hermitian metric $h$ on a holomorphic line bundle, we defines $\ogr(h):=\ogr(\varphi)$ if $h$ can be written locally as $h=e^{-2\varphi}$.
Furthermore, for any quasi-plurisubharmonic function $\varphi:X\longrightarrow [-\infty,+\infty)$, the $\cal{O}_X$-sheaf $\ogr(\varphi)$ satisfies the functoriality property $\ogr(\varphi)=\pi_*(K_{\tx}\otimes\scr{I}(\pi^*\varphi))$ (see \cite[Proposition 5.8]{Dem12}), 
where $\pi:\tx\longrightarrow X$ is any resolusion of singularities. 
In particular, if $X$ is smooth, then $\ogr(\varphi)=K_X\otimes\scr{I}(\varphi)$ by definition and $\pi:\tx\longrightarrow X$ can be taken to be any proper modification.

Throughout this paper, all complex spaces are assumed to be reduced. 
Let $X$ be a pure-dimensional complex space, let $L\longrightarrow X$ be a holomorphic line bundle equipped with a \textit{singular} Hermitian metric $h$ (see Definition \ref{Definition: singular Hermitian metrics on cpx sp}), and let $E\longrightarrow X$ be a holomorphic vector bundle equipped with a \textit{smooth} Hermitian metric $h_E$. 
For a resolution of singularities $\pi:\tx\longrightarrow X$, denote by $\exc:=\pi^{-1}(\xs)$ the $\pi$-exceptional divisor with simple normal crossings. 
When $X$ is non-compact, the resolution $\pi$ is always taken to be the canonical desingularization given in Theorem \ref{Theorem: canonical desingularization}.

We first present the following $L^2$-Dolbeault fine resolution and isomorphism.
Let $\scr{L}^{p,q}_{L,h}$ denote the subsheaf on $X$ consisting of germs of $L$-valued $(p,q)$-forms $u$ with measurable coefficients such that both $|u|^2_h$ and $|\dbar u|^2_h$ are locally integrable on the regular locus (see Section \ref{Section: dbar-operator and L2-Dolbeault complexes}).  
Here $\dbar$-operator is considered in the sense of distribution on $\reg$. 
The subsheaf $\scr{L}^{p,q}_{E\otimes L,h_E\otimes h}$ is defined in the same manner.

\begin{theorem}\label{Theorem: fine Dolbeault resolution}
    Let $X$ be a complex space of pure dimension $n$ and $L\longrightarrow X$ be a holomorphic line bundle equipped with a singular Hermitian metric $h$.
    Let $E\longrightarrow X$ be a holomorphic vector bundle with a smooth Hermitian metric $h_E$.
    If any local weight function of $h$ on $X$ is quasi-plurisubharmonic, then the $L^2$-Dolbeault complex 
    \begin{align*}
        0\longrightarrow \ogr(h)\otimes\cal{O}_X(E\otimes L)\hookrightarrow\scr{L}^{n,0}_{E\otimes L,h_E\otimes h}\overset{\dbar}{\longrightarrow}\scr{L}^{n,1}_{E\otimes L,h_E\otimes h}\overset{\dbar}{\longrightarrow}\scr{L}^{n,2}_{E\otimes L,h_E\otimes h}\overset{\dbar}{\longrightarrow}\cdots
    \end{align*}
    is exact; that is, the complex $(\scr{L}^{n,\ast}_{E\otimes L,h_E\otimes h},\dbar)$ is $L^2$-Dolbeault fine resolution of $\ogr(h)\otimes\cal{O}_X(E\otimes L)$.
    Thus, we have the following $L^2$-Dolbeault isomorphism 
    \begin{align*}
        H^q(X,\ogr(h)\otimes\cal{O}_X(E\otimes L))\cong H^q(\Gamma(X,\scr{L}^{n,\ast}_{E\otimes L,h_E\otimes h}))
    \end{align*}
    for any $q>0$.
\end{theorem}

\begin{theorem}\label{Theorem: L2-Dolbeault isomorphism}
    Let $X$ be a complex space of pure dimension $n$ and $\pi:\tx\longrightarrow X$ be a canonical desingularization. 
    If any local weight of $h$ on $X$ is quasi-plurisubharmonic, then the $L^2$-Dolbeault complex $(\pi_*\scr{L}^{n,*}_{\pi^*E\otimes\pi^*L,\pi^*h_E\otimes\pi^*h},\pi_*\dbar)$ 
    is a fine resolution of 
    \begin{align*}
        \ogr(h)\otimes\cal{O}_X(E\otimes L)\cong\pi_*(\omega_{\tx}\otimes\cal{O}_{\tx}(\pi^*E\otimes\pi^*L)\otimes\scr{I}(\pi^*h)).
    \end{align*}
    Thus, we have the following $L^2$-Dolbeault isomorphism 
    \begin{align*}
        H^q(X,\ogr(h)\otimes\cal{O}_X(E\otimes L))\cong H^q(\tx,K_{\tx}\otimes\pi^*E\otimes\pi^*L\otimes\scr{I}(\pi^*h)) 
    \end{align*}
    for any $q\geq0$. 
\end{theorem}

In the case of line bundles, Theorems \ref{Theorem: fine Dolbeault resolution} and \ref{Theorem: L2-Dolbeault isomorphism} are known for smooth Hermitian metrics by \cite{Rup14}, while analogous results for singular Hermitian metrics were obtained in \cite{SZ23}. 
Furthermore, for vector bundles twisted by such line bundles, an analogous result to Theorem \ref{Theorem: fine Dolbeault resolution} was obtained in \cite[Theorem 5.1]{SZ24} under the assumption that an $\bb{R}$-PVHS exists.
Theorems \ref{Theorem: fine Dolbeault resolution} and \ref{Theorem: L2-Dolbeault isomorphism} generalize these results to a more natural setting without assuming the existence of an $\bb{R}$-PVHS or the compactness of $X$.

Subsequently, to establish the vanishing theorem, we present the following global $L^2$-estimates on the regular locus of a weakly pseudoconvex complex space.
Here, weakly pseudoconvexity includes both compact and Stein cases and forms the broadest and highly significant class among the complex-geometric objects that can be treated.
In fact, every complex Lie group is always weakly pseudoconvex (see \cite{Kaz73}).

\begin{theorem}\label{Theorem: global L2-estimates}
    Let $X$ be a weakly pseudoconvex \kah complex space of pure dimension $n$. 
    We assume that $h_E$ is Demailly $m$-semi-positive on $\reg$ and 
    $\iO{L,h}\geq\varepsilon\omega$ on $\reg$ in the sense of currents for some semi-positive smooth function $\varepsilon:X\longrightarrow\bb{R}_{\geq0}$ and a \kah metric $\omega$. 
    For any $q>0$ satisfying $m\geq\min\{n-q+1,\rom{rank}\,E\}$ and for any $f\in L^{2,loc}_{n,q}(X,E\otimes L,h_E\otimes h)$ satisfying $\dbar f=0$ on $\reg$ and $\int_X\frac{1}{\varepsilon}|f|^2_{h_E\otimes h,\omega}\,\dvo<+\infty$, 
    there exists $u\in L^2_{n,q-1}(X,E\otimes L;\omega,h_E\otimes h)$ satisfying $\dbar u=f$ on $\reg$ and 
    \begin{align*}
        \int_X|u|^2_{h_E\otimes h,\omega}\,\dvo\leq\frac{1}{q}\int_X\frac{1}{\varepsilon}|f|^2_{h_E\otimes h,\omega}\,\dvo.
    \end{align*}
    
    Furthermore, if $h_E$ is Nakano semi-positive on $\reg$, then the global $L^2$-estimate holds for every $q>0$. 
    Under the same assumption, even if $X$ is not \kah and $\omega$ is simply a Hermitian metric, the global $L^2$-estimate still holds only in the case $q=1$.
\end{theorem}

Together with Theorem \ref{Theorem: fine Dolbeault resolution}, we obtain the following.

\begin{theorem}\label{Theorem: Nadel vanishing on w.p.c. cpx sp}
    Let $X$ be a weakly pseudoconvex \kah complex space of pure dimension. 
    If $h$ is singular positive on $X$ and $h_E$ is Nakano semi-positive on $\reg$ in the usual sense, then we have the following vanishing 
    \begin{align*}
        H^q(X,\ogr(h)\otimes\cal{O}_X(E\otimes L))=0 
    \end{align*}
    for any $q>0$. 
    Furthermore, if $X$ is not necessarily K\"{a}hler, then we have only the first cohomology vanishing $H^1(X,\ogr(h)\otimes\cal{O}_X(E\otimes L))=0$.
\end{theorem}

Even without the existence of a \kah metric, we obtain the vanishing of first cohomology, which plays an important role in applications such as proving the existence of global holomorphic sections.
Taking $E$ to be the trivial bundle $\cal{O}_X$, this recovers Nadel vanishing and provides the first precise proof of Nadel vanishing. 
In fact, Nadel vanishing on weakly pseudoconvex complex spaces had already been mentioned in \cite[Remark 5.17]{Dem12}. 
However, it seems unlikely that it could have been proved with the techniques available at that time, since $L^2$-Dolbeault resolutions on complex spaces, even in the case of smooth Hermitian metrics, was established only later (see \cite{Rup14}).

When the complex space is (relative) compact, by applying Theorem \ref{Theorem: L2-Dolbeault isomorphism} and using the strong openness property and Demailly's approximation theorem, we can obtain the higher cohomology vanishing without assuming the existence of a \kah metric, which is a generalization of Nadel vanishing. 
In other words, for a complex space on which Nadel vanishing holds, neither normality nor projectivity is required; it is sufficient to assume that the space is Moishezon, which follows from the existence of a big line bundle (see \cite{Moi66}).
Unlike the case of complex manifolds, it should be noted that a Moishezon complex space is not necessarily projective even if it admits a \kah metric (see \cite{Moi75}, \cite[Theorem 6]{Nam02}).

\begin{theorem}\label{Theorem: Nadel vanishing on Moishezon}
    Let $X$ be a compact complex space of pure dimension $n$ and $L\longrightarrow X$ be a holomorphic line bundle. 
    If $L$ is big, then $X$ is Moishezon and there exists a singular Hermitian metric $h$ on $L$ such that $\iO{L,h}\geq\gamma$ on $\reg$ in the sense of currents for some Hermitian metric $\gamma$ on $X$, 
    and we have the following cohomology vanishing 
    \begin{align*}
        H^q(X,\ogr(h)\otimes\cal{O}_X(L))=0 
    \end{align*} 
    for any $q>0$, without assuming the existence of a \kah metric. 
    Here, if $X$ is normal, then we can take the singular Hermitian metric $h$ with $\iO{L,h}\geq\gamma$ on $X$ in the sense of currents, i.e., $h$ is singular positive. 
    
    Furthermore, the same vanishing holds after twisting by a holomorphic vector bundle $E\longrightarrow X$ with a smooth Hermitian metric that is Nakano semi-positive on $\reg$.
\end{theorem}

As an application, we obtain the following Kawamata-Viehweg type vanishing theorem without assuming conditions on the singularities such as klt, 
or projectivity, in particular the existence of a \kah metric.

\begin{corollary}\label{Corollary: Kawamata-Viehweg type vanishing}
    Let $X$ be a compact Moishezon space of pure dimension $n$ and $L\longrightarrow X$ be a holomorphic line bundle. 
    If $L$ is nef and big, then we have the following vanishing 
    \begin{align*}
        H^q(X,\ogr\otimes\cal{O}_X(L))=0 
    \end{align*}
    for any $q>0$.
    Furthermore, the same vanishing holds after twisting by a holomorphic vector bundle with a smooth Hermitian metric that is Nakano semi-positive on $\reg$.
\end{corollary}

For a normal complex space $X$, if $X$ or the pair $(X,\Delta)$ is klt, then $X$ has rational singularities and hence $\omega_X=\ogr$; 
therefore, Corollary \ref{Corollary: Kawamata-Viehweg type vanishing} does not require projectivity and gives a generalization of the Kawamata-Viehweg vanishing theorem, where $\omega_X$ is the dualizing sheaf of $X$.

\section{Preliminaries}

\subsection{Canonical desingularization of complex spaces}

Even in the non-compact case, a global resolution of singularities can be obtained, which is locally given by a finite sequence of blow-ups with smooth centers. 
This is achieved by patching together the resolutions of singularities constructed on relatively compact subsets.

\begin{theorem}[{\cite[Theorem\,13.3 and 13.4]{BM97}, cf.\,\cite{Hir64}}]\label{Theorem: canonical desingularization}
    Let $X$ be a complex space which is not necessarily compact or reduced. 
    There exists a desingularization $\pi:\tx\longrightarrow X$, which is a composite of a locally finite sequence of blow-ups, such that 
    \begin{itemize}
        \item the map $\pi:\tx\longrightarrow X$ is proper holomorphic.
        \item the set $\tx$ is smooth and the $\pi$-exceptional set $\exc:=\pi^{-1}(\xs)$ is simple normal crossing, where $\exc$ denotes the collection of all exceptional divisors.
        \item for any relatively compact open subset $V$ of $X$, the restriction $\pi|_V:\tx|_{\pi^{-1}(V)}\longrightarrow X|_V$ is a composite of a finite sequence of blow-ups with smooth centres. 
        \item the restriction $\pi|_{\tx\setminus \exc}:\tx\setminus \exc\longrightarrow X\setminus\xs=\reg$ is biholomorphic.
        \item $\pi$ is canonical in the sense that for any isomorphic $\varphi:X|_U\overset{\cong}{\longrightarrow} X|_V$, where $U$ and $V$ are open subsets of $X$, lifts to an isomorphism $\widetilde{\varphi}:\tx|_{\pi^{-1}(U)}\longrightarrow\tx|_{\pi^{-1}(V)}$.
    \end{itemize}
\end{theorem}

The desingularization in this theorem is referred to as the \textit{canonical} \textit{desingularization}.

\subsection{Singular Hermitian metrics and its positivity on complex spaces}

Let $X$ be a complex space. 

\begin{definition}[{\cite[Chapter\,V, Definition\,1.4]{GPR94}, \cite{Fuj75}}]
    A function $\varphi:X\longrightarrow[-\infty,+\infty)$ is called (resp. \textit{strictly}) \textit{plurisubharmonic} if for any $x\in X$ 
    there exist an open neighborhood $U$ admitting a closed holomorphic embedding $\iota_U:U\hookrightarrow V\subset\bb{C}^N$, here $V$ is an open subset of $\bb{C}^N$, 
    and a (resp. strictly) plurisubharmonic function $\widetilde{\varphi}:V\longrightarrow[-\infty,+\infty)$ such that $\varphi|_U=\widetilde{\varphi}\circ\iota_U$.
\end{definition}

The same requirement is imposed on smooth functions, differential forms, test forms, Hermitian metrics, and \kah metrics on $X$: locally, they are required to be expressible, via a closed holomorphic embedding $\iota_U:U\hookrightarrow V\subset\bb{C}^N$, as the restrictions of the corresponding smooth functions, differential forms, and related objects on defined on $V$.
Moreover, a function is said to be \textit{quasi}-\textit{plurisubharmonic} if it can be written locally as the sum of a smooth function and a plurisubharmonic function.

Let $L\longrightarrow X$ be a holomorphic line bundle. Then there exist an open covering $\{U_\alpha\}_{\alpha\in\Lambda}$ of $X$ and isomorphisms $\iota_\alpha:L|_{U_\alpha}\overset{\cong}{\longrightarrow}U_\alpha\times\bb{C}$. 
Conversely, the set of holomorphic functions $\{f_{\alpha\beta}\}_{\alpha,\beta\in\Lambda}$, where $f_{\alpha\beta}:=\iota_\alpha\circ\iota_\beta^{-1}|_{U_\alpha\cap U_\beta}\in \varGamma(U_\alpha\cap U_\beta,\cal{O}_X)$, defines $L$. 
Such a $(\{f_{\alpha\beta}\},\{U_\alpha\})_{\Lambda}$ is called a \textit{system of transition functions of} $L$.

\begin{definition}[{\cite{Fuj75,Wat25b}}]
    We say that $h$ is \textit{smooth Hermitian metric} on $L$ if for a system of transition functions $(\{f_{\alpha\beta}\},\{U_\alpha\})_{\Lambda}$ of $L$, there is a collection of positive smooth functions $h=\{h_\alpha:U_\alpha\longrightarrow\bb{R}_{>0}\}_{\alpha\in\Lambda}$ such that $h_\beta/h_\alpha=|f_{\alpha\beta}|^2$ on $U_\alpha\cap U_\beta$.
\end{definition}

A \textit{smooth Hermitian metric} on a holomorphic vector bundle is defined similarly.
For any trivialization $\tau:L|_U\overset{\cong}{\longrightarrow}U\times\bb{C}$, a Hermitian metric $h$ on $L$ can be expressed as 
\begin{align*}
    ||\xi||_h=|\tau(\xi)|e^{-\varphi(x)}, \quad x\in U, \,\xi\in L_x
\end{align*}
using a function $\varphi$ on $U$. 
The function $\varphi:U\longrightarrow\bb{R}$ is called the \textit{weight function of} $h$ \textit{with respect to the trivialization} $\tau$.

\begin{definition}[{\cite{Fuj75,Wat25b}}]
    A smooth Hermitian metric $h$ on $L$ is said to be \textit{positive} if there exists a collection of positive smooth functions $h=\{h_\alpha\}_{\alpha\in\Lambda}$ on $L$ such that $-\log h_\alpha$ is strictly plurisubharmonic on $U_\alpha$ for all $\alpha\in\Lambda$. 
    In other words, the weight function of $h$ with respect to any trivialization $\tau$ is smooth strictly plurisubharmonic. 

    A holomorphic line bundle $L\longrightarrow X$ is said to be \textit{positive} if there exists a smooth Hermitian metric $h$ on $L$ which is positive. 
\end{definition}

\begin{definition}\label{Definition: singular Hermitian metrics on cpx sp}
    Let $X$ be a complex space and $L\longrightarrow X$ be a holomorphic line bundle. 
    We say that $h$ is a singular Hermitian metric on $L$ if for any smooth Hermitian metric $h_0$ on $L$, there exists a locally integrable function $\varphi$ on $X$, i.e., $\varphi\in L^1_{loc}(X)$, such that $h=h_0e^{-2\varphi}$ on $X$.
\end{definition}

For any open subset $U\subset X$, we denote 
\begin{align*}
    L^1_{loc}(U):=\biggl\{u:U\longrightarrow\bb{C} \text{ measurable}\,\bigg|\,\int_{K_{reg}}|u|\,dV<+\infty, \forall\,K\Subset U\biggr\}
\end{align*}
and note that, in general, $L^1_{loc}(U)\subsetneq L^1_{loc}(\ureg)$.

\begin{definition}\label{Definition: singular positivity on cpx sp}
    Let $X$ be a complex space and $h$ be a singular Hermitian metric on a holomorphic line bundle $L\longrightarrow X$. We say that a singular Hermitian metric $h$ is 
    \begin{itemize}
        \item \textit{singular} \textit{semi}-\textit{positive} if the weight function of $h$ with respect to any trivialization coincides with a plurisubharmonic function almost everywhere. 
        \item \textit{singular} \textit{positive} if the weight function of $h$ with respect to any trivialization coincides with a strictly plurisubharmonic function almost everywhere. 
    \end{itemize}
\end{definition}

When $X$ is smooth, Definitions \ref{Definition: singular Hermitian metrics on cpx sp} and \ref{Definition: singular positivity on cpx sp} coincide with the existing definitions (see \cite[Chapter 3]{Dem12}, \cite{Wat25a,Wat25b}).

\subsection{Currents and its positivity on complex space}

Let $X$ be a complex space of pure dimension $n$. 
As in the smooth case, the sheaf $\scr{D}'^{\,p,q}_X$ of currents of bidegree $(p,q)$ on $X$ is by definition the dual of $\scr{D}^{\,p,q}_X$ which is the space of smooth differential forms of bidegree $(p,q)$ with compact support. 
Given a local embedding $\iota:U\hookrightarrow \bb{C}^N$, thus, the currents $T\in\scr{D}'^{\,p,q}_X$ precisely correspond, via $T\longmapsto\iota_*T$, to the currents of bidegree $(N-n+p,N-n+q)$ in the ambient space that vanish on all test forms $\phi$ such that $\iota^*\phi=0$ on $\reg$ (see \cite{Dem85,AS12}).
Here, the current $\iota_*T$ is defined by 
\begin{align*}
    \lara{\iota_*T}{\tl{\phi}}:=\lara{T}{\iota^*\tl{\phi}}=\int_{U_{reg}} T\wedge\iota^*\tl{\phi},
\end{align*}
for any test forms $\tl{\phi}$ on $\bb{C}^N$. 
Furthermore, the positivity of the current $T\in\scr{D}'^{\,p,p}_X$ is also defined as in the smooth case, and this is equivalent to the positivity of $\iota_*T$. 
In other words, a current $T\in\scr{D}'^{\,p,p}_X$ is said to be \textit{positive}, i.e., $T\geq0$ \textit{on} $X$ \textit{in the sense of currents}, if $\lara{T}{\phi}\geq0$ for any strongly positive test form $\phi\in\scr{D}^{n-p,n-p}_X(X)$ on $X$, which can be written as $\phi=\rho\cdot i^{(n-p)^2}\eta\wedge\overline{\eta}$ suitable forms $\rho\in\scr{D}(X,\bb{R}_{\geq0})$ and $\eta\in\cal{E}_X^{n-p,0}(X)$.
Here, $\cal{E}_X^{p,q}$ denotes the sheaf of smooth differential forms. 
Note that, if $\varphi\in L^1_{loc}(U)$ then $[\varphi]\in\scr{D}_X'(U)$ and $\iota_*[\varphi]\in\scr{D}'(\bb{C}^N)$; however, in general $\iota_*\varphi$ is not locally integrable. 
Moreover, the form $\idd\varphi$ is defined on $U$ as a current by 
\begin{align*}
    \lara{\idd\varphi}{\phi}:=\int_{U_{reg}}\varphi\,\idd\phi,
\end{align*}
for any test form $\phi\in\scr{D}^{n-1,n-1}(U)$, i.e., $\idd\varphi\in\scr{D}_X'^{\,1,1}(U)$.  
In particular, it follows that the curvature $\iO{L,h}$ associated with a singular Hermitian metric $h$ on $L$ is always well-defined as a current, 
and the current $\iO{L,h}$ always has locally finite mass, i.e., 
\begin{align*}
    ||\iO{L,h}||_{\omega}(K):=\int_K\iO{L,h}\wedge\omega^{n-1}=\sum_{\alpha\in\Lambda}\int_{U_\alpha\cap K}\varphi_\alpha\,\idd\omega^{n-1}<+\infty,
\end{align*}
for any Hermitian metric $\omega$ and any compact subset $K\Subset X$, here $\varphi_\alpha\in L^1_{loc}(U_\alpha)$ is local weight function of $h$.
Therefore, the following holds by the Skoda-El Mir extension theorem (see \cite[Chapter III, Theorem 2.3]{Dem-book}, \cite[Theorem 1.18]{Dem12}).

\begin{proposition}
    Let $X$ be a complex space of pure dimension, $\pi:\tl{X}\longrightarrow X$ be a canonical desingularization and $\exc$ be the $\pi$-exceptional set.
    Let $L\longrightarrow X$ be a holomorphic line bundle with a singular Hermitian metric $h$. 
    If each local weight function of $h$ near $\xs$ is quasi-plurisubharmonic, then the pullback $\pi^*\iO{L,h}$, which is defined as a current on $\tx\setminus \exc$, has locally bounded mass near $\exc$ and admits a trivial extension to $\tx$.
    Hence, it coincides with the curvature current $\iO{\pi^*L,\pi^*h}$ on $\tx$ associated with the pulled-back metric $\pi^*h$ in the sense of currents, i.e., $\pi^*\iO{L,h}=\iO{\pi^*L,\pi^*h}$ on $\tx$. 
\end{proposition}

Note that, in general, it is unclear whether $dd^cu$ can be defined on $U$ as a current under the only assumption that $u\in L^1_{loc}(\ureg)$.
Here, a closed $(1,1)$-current $T$ on a complex space $X$ always admits local potentials on $\reg$, but it does not necessarily admit local potentials on the whole space $X$. 
For instance, this occurs for the integration current associated with an effective Weil divisor that is not $\bb{Q}$-Cartier.

\subsection{Multiplier ideal sheaves and the strong openness property}

In this subsection, we consider the smooth case, and let $X$ be a complex manifold.

\begin{definition}
    Let $\varphi:X\longrightarrow[-\infty,+\infty)$ be a quasi-plurisubharmonic function.
    The \textit{multiplier} \textit{ideal} \textit{sheaf} $\scr{I}(\varphi)\subset\cal{O}_X$ is defined by 
    \begin{align*}
        \scr{I}(\varphi)(U):=\{f\in\cal{O}_X(U)\mid |f|^2 e^{-2\varphi}\in L^1_{loc}(U)\}
    \end{align*}
    for any open subset $U\subset X$. For a singular Hermitian metric $h$ on $L$ with the local weight $\varphi$, i.e., $h=h_0e^{-2\varphi}$, we define the multiplier ideal sheaf of $h$ by $\scr{I}(h):=\scr{I}(\varphi)$.
\end{definition}

We introduce the strong openness property, which is also important for applications.

\begin{theorem}[\textnormal{Strong openness property, cf. \cite{GZ15}}]\label{Theorem: strong openness property}
    Let $\varphi$ be a negative plurisubharmonic function on $\Delta^n\subset\bb{C}^n$, and let $\psi\not\equiv-\infty$ be a negative plurisubharmonic function on $\Delta^n$. 
    Then we have $\scr{I}(\varphi)=\bigcup_{\varepsilon>0}\scr{I}(\varphi+\varepsilon\psi)$.
\end{theorem}

In the presence of singularities, it is not clear whether a multiplier ideal sheaf defined as in the smooth case satisfies the desired properties. 
On a complex space, taking singularities into account via Ohsawa's extension measure, the \textit{Nadel}-\textit{type} \textit{multiplier ideal sheaf} $\widehat{\scr{I}}(\varphi)$ (see \cite[Definition 3.2]{Li24}), which are known to satisfy the strong openness property (see \cite[Theorem 1.1]{Li24}), are introduced.
Recall that if the complex space $X$ is locally a complete intersection (not necessarily normal), then $\ogr(\varphi)=\omega_X\otimes\widehat{\scr{I}}(\varphi)$ (see \cite[Remark 3.3 (2)]{Li24}). 

\section{Positivity of curvature currents and bigness on compact space}

Singular positivity admits the following reformulation in terms of curvature currents. 

\begin{proposition}\label{Proposition: curvature condition of sHm}
    Let $X$ be a complex space of pure dimension $n$ and $L\longrightarrow X$ be a holomorphic line bundle with a singular Hermitian metric $h$. 
    We obtain the following. 
    \begin{itemize}
        \item If $h$ is singular semi-positive, then $\iO{L,h}\geq0$ on $X$ in the sense of currents. 
        \item If $h$ is singular positive, then for any Hermitian metric $\omega$ on $X$, there exists a positive smooth function $\varepsilon:X\longrightarrow\bb{R}_{>0}$ such that $\iO{L,h}\geq\varepsilon\omega$ on $X$ in the sense of currents. 
    \end{itemize}
    Furthermore, the converse also holds if $X$ is normal, or if any weight function is locally bounded from above in a neighborhood of $\xs$ and $X$ is locally irreducible.
\end{proposition}

\begin{proof}
    Since positivity is a local property, it suffices to choose an arbitrary neighborhood $U$ admitting a trivialization $\tau:L|_U\overset{\simeq}{\longrightarrow}U\times\bb{C}$ and a holomorphic embedding $\iota:U\hookrightarrow\bb{C}^N$, and to show that $\lara{\iO{L,h}}{\phi}\geq0$
    for any strongly positive test form $\phi\in\scr{D}_X^{n-1,n-1}(U)$, which can be written as $\phi=\rho\cdot i^{(n-1)^2}\eta\wedge\overline{\eta}$ for suitable forms $\rho\in\scr{D}(X,\bb{R}_{\geq0})$ and $\eta\in\cal{E}_X^{n-1,0}(U)$. 
    Let $\varphi$ be a weight function of $h$ with respect to a trivialization $\tau$ and $\pi:\tx\longrightarrow X$ be a canonical desingularization. 
    Since plurisubharmonicity is preserved under holomorphic maps (see \cite{GR56}), the curvature current $\iO{\pi^*L,\pi^*h}\geq0$ holds on $\tx$. 
    Therefore, we obtain 
    \begin{align*}
        \lara{\iO{L,h}}{\phi}&=\int_{\ureg} \idd\varphi\wedge\phi=\int_{\pi^{-1}(\ureg)}\pi^*\idd\varphi\wedge\pi^*\phi=\int_{\pi^{-1}(\ureg)}\idd\pi^*\varphi\wedge\pi^*\phi\\
        &=\int_{\pi^{-1}(\ureg)}\iO{\pi^*L,\pi^*h}\wedge\pi^*\phi=\lara{\iO{\pi^*L,\pi^*h}}{\pi^*\phi}\geq0.
    \end{align*}

    Conversely, assume that $\iO{L,h}\geq0$ on $X$. 
    Then, for any strongly positive test form $\tl{\phi}\in\scr{D}_{\tl{X}}^{n-1,n-1}(\pi^{-1}(\ureg))$ on $\pi^{-1}(\ureg)=\pi^{-1}(U)\setminus \exc$, the pullback $(\pi^{-1})^*\tl{\phi}\in\scr{D}_X^{n-1,n-1}(\ureg)$ is also strongly positive test form on $\ureg$, and hence the following holds. 
    \begin{align*}
        0\leq\lara{\iO{L,h}}{(\pi^{-1})^*\tl{\phi}}=\int_{\pi^{-1}(\ureg)}\iO{\pi^*L,\pi^*h}\wedge\tl{\phi}=\lara{\iO{\pi^*L,\pi^*h}}{\tl{\phi}}.
    \end{align*}
    Therefore, we have $\iO{\pi^*L,\pi^*h}\geq0$ on $\tx\setminus \exc$ in the sense of currents; in particular, the local weight function $\pi^*\varphi$, viewed as a local potential, is plurisubharmonic on $\pi^{-1}(\ureg)$. 
    Since $\pi^{-1}$ is holomorphic on $\ureg$, the weight function $\varphi$ is also plurisubharmonic on $\ureg$. 
    Furthermore, by \cite[Satz 3 and 4]{GR56} and \cite[Theorem 1.10]{Dem85}, if $X$ is normal, or if $\varphi$ is locally bounded from above on $U$, then $\varphi$ extends uniquely as a weakly plurisubharmonic function on $U$. 
    Finally, it follows from \cite[Theorem 5.3.1]{FN80} that every weakly plurisubharmonic function is plurisubharmonic.

    Second, we assume that $h$ is singular positive. 
    Since $X$ is second countable, it admits an increasing sequence $\{X_j\}_{j\in\bb{N}}$ of relatively compact open subsets, i.e., $X_j\Subset X_{j+1}$ for any $j$ and $\bigcup_{j\in\bb{N}}X_j=X$. 
    By compactness of $\overline{X_j}$, there exists a finite open covering $\{U_k\}_{1\leq k\leq N}$ such that $\overline{X_j}\subset\bigcup^N_{k=1}U_k$, $L|_{U_k}$ is trivial and each $U_k$ admits a holomorphic embedding $\iota_k:U_k\hookrightarrow\bb{C}^{N_k}$.
    Here, the weight of $h$ on each $U_k$ coincides with a strictly plurisubharmonic function $\varphi_k$ a.e., and for each embedding $\iota_k$, there exists a strictly plurisubharmonic function $\widetilde{\varphi}_k$ on $V_k$ and a Hermitian metric $\widetilde{\omega}_k$ on $V_k$ 
    such that $\varphi_k|_{U_k}=\iota_k^*\widetilde{\varphi}_k$ and $\omega|_{U_k}=\iota_k^*\widetilde{\omega}_k$ on $U_k$, where $V_k\Subset\bb{C}^{N_k}$ is open subsets with $\iota_k(U_k)\Subset V_k$.
    By strictly plurisubharmonicity of $\widetilde{\varphi}_k$, there exists a constant $\delta_k>0$ such that $\idd\widetilde{\varphi}_k\geq \delta_k\idd\sum^{N_k}_{j=1}|z_j^{(k)}|^2=\delta_k\idd|z^{(k)}|^2$ in the sense of currents on $V_k$ for a standard coordinates $(z_1^{(k)},\ldots,z_{N_k}^{(k)})$ of $\bb{C}^{N_k}$, 
    then we have $\iO{L,h}=\idd\varphi_k\geq\delta_k\iota_k^*\idd|z^{(k)}|^2$ on $U_k$ in the sense of currents. Here, there exists a constant $c_k>0$ such that $\idd|z^{(k)}|^2\geq c_k\widetilde{\omega}_k$ on $V_k$, hence $\iota_k^*\idd|z^{(k)}|^2\geq c_k\omega$ on $U_k$. 
    Combined with the above inequality, setting $\varepsilon_{U_k}:=\delta_kc_k>0$, we obtain $\iO{L,h}=\idd\varphi_k\geq \varepsilon_{U_k}\omega$ on $U_k$ in the sense of currents. 
    Let $\varepsilon_j:=\min_k\varepsilon_{U_k}>0$, then we can construct a smooth function $\varepsilon:X\longrightarrow\bb{R}_{>0}$ satisfying $\varepsilon_j>\varepsilon>0$ on $X_j\setminus\overline{X}_{j-1}$. 
    The converse is clear from the above proof.
\end{proof}

Let $X$ be a complex space and $L\longrightarrow X$ be a holomorphic line bundle. 
We define $L$ to be \textit{big} if the Kodaira-Iitaka dimension of $L$ is maximal on each irreducible component $X_j$, i.e., $\kappa(X_j,L|_{X_j})=\dim X_j$. 
Here, the Kodaira-Iitaka dimension is defined in the same way without assuming compactness (see \cite[Definition 2.4]{Wat24}).
A compact complex space $X$ is said to be \textit{Moishezon} if the algebraic dimension of each irreducible component $X_j$ is maximal (cf. \cite{Moi66}), i.e., $a(X_j)=\dim X_j$.
A Moishezon space is bimeromorphic to a projective manifold, as follows:

\begin{theorem}[{cf. \cite{Moi66,Hir64,GR70}, \cite[Chapter VII, Corollary 6.10, Theorem 6.14 and Proposition 6.16]{GPR94}}]\label{Theorem: Moishezon-ness}
    Let $X$ be a compact complex space. If $X$ is Moishezon, then there exists a bimeromorphic modification $\pi:\tx\longrightarrow X$, given by a composition of finitely many blow-ups with smooth center, 
    such that $\tx$ is projective manifold and the restriction map $\pi|_{\tx\setminus\exc}:\tx\setminus\exc\overset{\simeq}{\longrightarrow}X\setminus Z$ is biholomorphic, where $Z$ is an analytic subset of $X$ with $\xs\subset Z$ and $\exc:=\pi^{-1}(Z)$ is the $\pi$-exceptional divisor with simple normal crossing.
    Furthermore, $X$ is Moishezon if and only if $X$ carries an almost positive torsion-free coherent sheaf $\scr{L}$ with $\rom{supp}\,\scr{L}=X$.
    
    In particular, if $X$ carries a big line bundle, then $X$ is Moishezon.
\end{theorem}

Thus, since a compact Moishezon space contains sufficiently many curves, we define a line bundle to be \textit{nef} as usual in terms of the non-negativity of its intersection numbers.
Furthermore, nefness is preserved under pullbacks by resolution of singularities.

A Moishezon manifold is projective if and only if it admits a \kah metric (see \cite{Moi66}). 
Note that, in the case of Moishezon spaces, the existence of a \kah metric does not in general imply projectivity (see \cite{Moi75}); however, a \kah Moishezon space with 1-rational singularities is known to be projective (see \cite[Theorem 6]{Nam02}).
Here, a complex space $X$ with $1$-\textit{rational singularities} is normal and admits a resolusion $\pi:Y\longrightarrow X$ such that $R^1\pi_*\cal{O}_Y=0$; rational singularities are examples of this.

\begin{theorem}\label{Theorem: characterization of big and singular positive}
    Let $X$ be a compact complex space and $L\longrightarrow X$ be a holomorphic line bundle. 
    We have the following relationship.
    \begin{itemize}
        \item If $L$ is singular positive, then $L$ is big. 
        \item If $L$ is big, then there exists a singular Hermitian metric $h$ on $L$ such that $\iO{L,h}\geq\gamma$ on $\reg$ in the sense of currents for some Hermitian metric $\gamma$ on $X$.
        
        Furthermore, if $X$ is normal, then we can take the singular Hermitian metric $h$ with $\iO{L,h}\geq\gamma$ on $X$ in the sense of currents for some Hermitian metric $\gamma$ on $X$, i.e., $h$ is singular positive. 
    \end{itemize}
    Thus, if $X$ is normal, then singular positivity and bigness coincide.
\end{theorem}

\begin{proof}
    Let $\pi:\tx\longrightarrow X$ be a resolusion of singularities. 
    By Proposition \ref{Proposition: curvature condition of sHm}, if $L$ is singular positive on $X$, then there exist a singular Hermitian metric $h$ on $L$ and a Hermitian metric $\omega$ on $X$ such that $\iO{L,h}\geq\omega$ holds on $X$ in the sense of currents. 
    By applying \cite[Lemma 3.2]{Wat25b}, there exist a quasi-plurisubharmonic function $\psi:\tx\longrightarrow[-\infty,+\infty)$ which is smooth on $\tx\setminus\exc$ and a number $\varepsilon_X>0$ such that $\pi^*\omega+\varepsilon\idd\psi$ is strictly positive on $\tx$ in the sense of currents for any $0<\varepsilon<\varepsilon_X$; 
    that is, there exists a Hermitian metric $\tl{\gamma}_\varepsilon>0$ on $\tx$ such that $\pi^*\omega+\varepsilon\idd\psi\geq\tl{\gamma}_\varepsilon$ on $\tx$ in the sense of currents. 
    Therefore, defining a singular Hermitian metric on $\pi^*L$ by $\pi^*he^{-\varepsilon\psi}$, it becomes singular positive on $\tx$. Hence, by Demailly's characterization of bigness (see \cite{Dem90} \cite[Chapter 6]{Dem12}), $\pi^*L$ is big.
    Since bigness is a bimeromorphic invariant, $L$ is also big.

    If $L$ is big, then $\pi^*L$ is also big and admits a singular positive Hermitian metric $\tl{h}$, i.e., $\iO{\pi^*L,\tl{h}}\geq\tl{\omega}$ on $\tx$ in the sense of currents for some Hermitian metric $\tl{\omega}$. 
    Here, there exists a Hermitian metric $\gamma$ on $X$ such that $\tl{\omega}\geq\pi^*\gamma$ on $\tx$. 
    We define a singular Hermitian metric $h$ of $L$ by setting $h=(\pi^{-1})^*\tl{h}$ on $\reg$ and $h=0$ on $\xs$. 
    By the biholomorphicity of $\pi|_{\tx\setminus \exc}:\tx\setminus \exc\overset{\simeq}{\longrightarrow}\reg$, we obtain that $\iO{L,h}\geq\gamma$ on $\reg$ in the sense of currents. 
    For this singular Hermitian metric $h$, we have $\iO{\pi^*L,\pi^*h}\geq\pi^*\gamma$ on $\tx\setminus \exc$ in the sense of currents. Then, if $X$ is normal, as in the proof of Proposition \ref{Proposition: curvature condition of sHm}, 
    by uniquely replacing the weight function uniquely via the extension of plurisubharmonic functions (see \cite[Satz 4]{GR56}), we can adjust the values of $h$ on $\xs$ to obtain a singular positive Hermitian metric.
\end{proof}

This proof yields the following.

\begin{lemma}\label{Lemma: pullback of big for singular Hermitian metric}
    Let $X$ be a compact complex space and $L\longrightarrow X$ be a holomorphic line bundle. 
    If $L$ is big, then there exists a resolusion of singularities $\pi:\tx\longrightarrow X$ such that $\tx$ is projective and $\pi^*L$ is big. 
    Furthermore, $\pi^*L$ admits a singular Hermitian metric $\tl{h}$ on $\tx$, which is singular positive and can be choosen such that $\scr{I}(\tl{h})=\scr{I}(\pi^*h)$ on $\tx$ 
    for the singular Hermitian metric $h$ on $L$ appearing in Theorem \ref{Theorem: characterization of big and singular positive}.
\end{lemma}

\section{Global $L^2$-estimates on the regular locus $\reg$}

In this section, we provide a proof of $L^2$-existence Theorem \ref{Theorem: global L2-estimates}.
Let $X$ be a complex space of pure dimension $n$ and $L$ be a holomorphic line bundle on $X$ with a singular Hermitian metric $h$. 
For any open subset $U\subset X$, we denote by 
\begin{align*}
    L^{2,loc}_{p,q}(U,L,h):=\Bigl\{u\in L^{2,loc}_{p,q}(\ureg,L,h)\,\Big|\, u|_{K_{reg}}\in L^2_{p,q}(K_{reg},L,h),\forall\,K\Subset U\Bigr\}
\end{align*} 
the space of $L$-valued $(p,q)$-forms on $U$ which are locally square integrable with respect to $h$. 
This space does not depend on the choice of a Hermitian metric on $U$; more precisely, for any Hermitian metric $\omega_U$ on $U$, we have $L^2_{p,q}(K_{reg},L,h)=L^2_{p,q}(K_{reg},L;\omega_U,h)$.
We further denote by 
\begin{align*}
    L^2_{p,q}(X,L;\omega,h)=&\,L^2_{p,q}(\reg,L;\omega,h)\\
    :=&\biggl\{u\in L^{2,loc}_{p,q}(X,L,h)\,\bigg|\,||u||^2_{h,\omega}:=\int_{\reg}|u|^2_{h,\omega}dV_\omega<+\infty\biggr\}
\end{align*}
the space of $L$-valued $(p,q)$-forms $u$ with measurable coefficients that are square integrable with respect to $h$ and a fixed smooth Hermitian metric $\omega$.
The same notion is defined for holomorphic vector bundles.
Unless otherwise stated, integrals over $\reg$ will be denoted simply by integrals over $X$. 
We first introduce several notions of positivity. 

\begin{definition}[{\cite[Chapter VII]{Dem-book}, \cite{Dem82}}]
    Let $T$ and $V$ be complex vector spaces of dimensions $n,r$ respectively, and $\Theta$ be a Hermitian form on $T\otimes V$. 
    Let $F$ be a holomorphic vector bundle over a complex manifold $M$. 
    \begin{itemize}
        \item A tensor $u\in T\otimes V$ is said to be of \textit{rank} $m$ if $m\geq0$ is the smallest integer such that $u$ can be written $u=\sum^m_{j=1}\xi_j\otimes s_j$, where $\xi_j\in T$, $s_j\in V$. 
        \item $\Theta$ is $m$-\textit{positive} (resp. $m$-\textit{semi}-\textit{positive}) if $\Theta(u)>0$ (resp. $\geq0$) for any tensor $0\ne u\in T\otimes V$ of rank $\leq m$. In this case, we write $\Theta>_m0$ (resp. $\geq_m0$). 
        \item A smooth Hermitian metric $h_F$ on $F$ is \textit{Demailly} $m$-\textit{positive} (resp. $m$-\textit{semi}-\textit{positive}) if $\theta_{F,h_F}>_m0$ (resp. $\theta_{F,h_F}\geq_m0$), where $\theta_{F,h_F}$ denotes the Hermitian form on $T_X\otimes F$ naturally associated with the curvature $\iO{F,h_F}$.
    \end{itemize}
\end{definition}

For simplicity, we write $\iO{F,h_F}$ in place of $\theta_{F,h_F}$ whenever no confusion is likely to arise.
Being $1$-positive is equivalent to Griffiths positivity, while Demailly $m$-positivity for $m\geq\min\{n,\rom{rank}\,F\}$ corresponds to Nakano positivity.
The following $L^2$-estimate due to Demailly plays an essential role in the proof of Theorem \ref{Theorem: global L2-estimates}.

\begin{theorem}[{\cite[Th\'{e}or\`{e}m 5.1]{Dem82}}]\label{Theorem: [Theorem 5.1, Dem82]}
    Let $M$ be a complex manifold of dimension $n$ which admits a complete \kah metric. Let $\varphi$ be a quasi-plurisubharmonic function on $M$ and $F\longrightarrow M$ be a holomorphic vector bundle with a smooth Hermitian metric $h_F$. 
    The Lebesgue decomposition of the order-zero current $\idd\varphi$ is therefore of the form $\idd\varphi=(\idd\varphi)_{ac}+(\idd\varphi)_{sing}$, where the singular part $(\idd\varphi)_{sing}$ is a nonnegative $(1,1)$-current, and the absolutely continuous part $(\idd\varphi)_{ac}$ is a $(1,1)$-form with $L^1_{loc}$-coefficients which is locally bounded from below.
    We assume that 
    \begin{align*}
        \iO{F,h_F}+(\idd\varphi)_{ac}\geq_{n-q+1}0
    \end{align*}
    on $M$. 
    Then, for any \kah metric $\omega$, not necessarily complete, and any $f\in L^{2,loc}_{n,q}(M,F,$ $h_F e^{-\varphi})$ satisfying $\dbar f=0$ and $\int_M\lara{[\iO{F,h_F}+(\idd\varphi)_{ac},\Lambda_\omega]^{-1}f}{f}_{h_F,\omega}e^{-\varphi}dV_\omega<+\infty$, 
    there exists $u\in L^2_{n,q-1}(M,F;\omega,h_Fe^{-\varphi})$ such that $\dbar u=f$ on $M$ and 
    \begin{align*}
        \int_M|u|^2_{h_F,\omega}e^{-\varphi}dV_\omega\leq\int_M\lara{[\iO{F,h_F}+(\idd\varphi)_{ac},\Lambda_\omega]^{-1}f}{f}_{h_F,\omega}e^{-\varphi}dV_\omega.
    \end{align*}
\end{theorem}

We briefly comment on this theorem. Its proof is obtained by an effective application of the following approximation theorem for quasi-plurisubharmonic functions. 

\begin{theorem}[{\cite[Th\'{e}or\`{e}m 9.1]{Dem82}}]
    Let $M$ be a complex manifold equipped with a \kah metric $\omega$ and $\varphi$ be a quasi-plurisubharmonic function on $M$. 
    Assume that there exists a continuous real $(1,1)$-form $\theta$ such that $\idd\varphi\geq\theta$ in the sense of currents. 
    Then there exist a decreasing family $(\varphi_\nu)_{\nu\in\bb{N}}$ of smooth functions on $M$, a family $(\gamma_\nu)_{\nu\in\bb{N}}$ of continuous real $(1,1)$-forms, and a decreasing family $(\lambda_\nu)_{\nu\in\bb{N}}$ of continuous functions on $M$ satisfying the following properties: 
    \begin{itemize}
        \item [$(a)$] $\displaystyle\lim_{\nu\to+\infty}\varphi_\nu(x)=\varphi(x)$ for every $x\in M$, 
        \item [$(b)$] $\idd\varphi_\nu\geq\gamma_\nu-\lambda_\nu\omega$ and $\gamma_\nu\geq\theta$, 
        \item [$(c)$] $\gamma_\nu\longrightarrow(\idd\varphi)_{ac}$ almost everywhere on $M$ as $\nu\to+\infty$,
        \item [$(d)$] $\lambda_\nu\longrightarrow0$ almost everywhere on $M$, more precisely, at every point $x\in M$ where the Lelong number $\nu(\varphi,x)=0$ vanishes,
        \item [$(e)$] If $\nu(\varphi,x)=0$ for every $x\in M$ (in particular, if $\varphi$ is locally bounded), then $\lambda_\nu$ converges uniformly to $0$ on every compact subset of $M$. 
    \end{itemize}
\end{theorem}

The proof of Theorem \ref{Theorem: [Theorem 5.1, Dem82]} proceeds as follows.
First, for each smooth approximation $\varphi_\nu$, we derive the desired $L^2$-estimate by applying H\"{o}rmander's elliptic-estimate with respect to the $n-q+1$-positive form $\theta$. 
The loss of positivity, measured by $\lambda_\nu$, is treated as an error term and is removed by passing to the limit as $\nu\to+\infty$.
Here, each $\lambda_\nu$ generally remains positive near the singular locus of $\varphi$, depending on the Lelong numbers, and therefore does not vanish pointwise. However, since $\lambda_\nu\to0$ almost everywhere on $M$, the error term disappears in the limit, which is the key ingredient in the argument.
Thus, condition $(d)$ is of fundamental importance. We should, however, note the following.

\begin{remark}
    At points where the Lelong number does not vanish, $\lambda_\nu$ remains strictly positive in a neighborhood of those points. 
    Of course, as $\nu\to+\infty$, $\lambda_\nu$ converges pointwise to $0$ on $M\setminus E_{+}(\varphi)=M\setminus\{x\in M\mid\nu(\varphi,x)>0\}$.
    However, this convergence does not extend to the values on $E_{+}(\varphi)$.
    Rather, the loss of positivity along $E_{+}(\varphi)$, measured by $\lambda_\nu$, persists with a definite amount of degeneracy, and $\lambda_\nu$ asymptotically behaves like a defining function for $E_{+}(\varphi)$.
    
    In particular, if $\varphi$ is merely strictly plurisubharmonic, one cannot expect that, for sufficiently large $\nu$, the approximation $\varphi_\nu$ becomes a smooth strictly plurisubharmonic function on the whole space $M$.
\end{remark}

\begin{remark}
    In fact, if \cite[Th\'{e}or\`{e}m 9.1]{Dem82} were able to construct a smooth strictly plurisubharmonic function from a strictly plurisubharmonic function, then a contradiction would arise.
\end{remark}

\begin{proof}
    Let $\pi:\tx\longrightarrow\bb{C}^n$ be the blow-up at a point $p$, and let $U$ be a relatively compact Stein neighborhood of $p$. 
    Then $\pi^{-1}(U)$ is a weakly pseudoconvex manifold, in particular, a $1$-convex manifold. However, it is not Stein, since it contains the exceptional divisor $E_p:=\pi^{-1}(x_0)\cong\bb{P}^{n-1}$ as a compact complex submanifold.
    Pulling back a smooth strictly plurisubharmonic function $\psi$ on $U$, the function $\pi^*\psi$ is a plurisubharmonic function that degenerates along $E_p$, while it is strictly plurisubharmonic on $\pi^{-1}(U)\setminus E_p$. 
    By the Negativity lemma (see \cite[Chapter VII, Proposition 12.4]{Dem-book}), there exists a smooth Hermitian metric $h_{E_p}^*$ on $\mathcal{O}_{\tx}(-E_p)$ induced by the normal bundle such that it compensates for the degeneracy of the positivity of $\idd\psi$ along $E_p$; namely, there exists $C>0$ such that $C\idd\pi^*\psi+\iO{\cal{O}_{\tx}(-E_p),h^*_{E_p}}>0$ on $\pi^{-1}(U)$.
    Hence, $h^*_{E_p}e^{-C\pi^*\psi}$ is a smooth positive Hermitian metric on $\cal{O}_{\tx}(-E_p)|_{\pi^{-1}(U)}$. Furthermore, consider the natural singular Hermitian metric $1/|\sigma|^2$ on $\cal{O}_{\tx}(E_p)$ induced by the defining section $\sigma$ of $E_p$. 
    This metric is singular semi-positive, and thus the singular Hermitian metric 
    \begin{align*}
        h^*_{E_p}e^{-C\pi^*\psi}\otimes\frac{1}{|\sigma|^2}=\frac{h^*_{E_p}}{|\sigma|^2}e^{-C\pi^*\psi}
    \end{align*}
    on $\cal{O}_{\tx}\cong\cal{O}_{\tx}(-E_p)\otimes\cal{O}_{\tx}(E_p)$ is singular positive. Therefore, the function
    \begin{align*}
        \varPsi:=-\log\Big(h_{E_p}^*e^{-C\pi^*\psi}\otimes\frac{1}{|\sigma|^2}\Big)=C\pi^*\psi-\log\frac{h_E^*}{|\sigma|^2}
    \end{align*}
    is strictly plurisubharmonic on $\pi^{-1}(U)$. In particular, locally we have $\Psi=\log|\sigma|^2+\text{smooth}$, and hence $E_p$ is the singular locus of $\varPsi$. 
    
    Applying \cite[Th\'{e}or\`{e}m 9.1]{Dem82} to this $\varPsi$, if $\lambda_\nu$ converges uniformly to $0$, then one can construct a smooth strictly plurisubharmonic function on $\pi^{-1}(U)$, which would imply that $\pi^{-1}(U)$ is Stein, leading to a contradiction.
\end{proof}

Here, a function $\varPsi:X\longrightarrow[-\infty,+\infty)$ on a complex space $X$ is \textit{exhaustion} if all sublevel sets $X_c:=\{x\in X\mid\varPsi(x)<c\}$, $\forall\,c\in\bb{R}$, are relatively compact. 
A complex space is said to be \textit{weakly} \textit{pseudoconvex} if there exists a smooth exhaustion plurisubharmonic function.

\begin{theorem}[{= Theorem\,\ref{Theorem: global L2-estimates}}]
    Let $X$ be a weakly pseudoconvex \kah complex space of pure dimension $n$ and $L\longrightarrow X$ be a holomorphic line bundle with a singular Hermitian metric $h$. 
    Let $E\longrightarrow X$ be a holomorphic vector bundle with a smooth Hermitian metric $h_E$.
    We assume that $h_E$ is Demailly $m$-semi-positive on $\reg$ and 
    $\iO{L,h}\geq\varepsilon\omega$ on $\reg$ in the sense of currents for some semi-positive smooth function $\varepsilon:X\longrightarrow\bb{R}_{\geq0}$ and a \kah metric $\omega$. 
    For any $q>0$ satisfying $m\geq\min\{n-q+1,\rom{rank}\,E\}$ and for any $f\in L^{2,loc}_{n,q}(X,E\otimes L,h_E\otimes h)$ satisfying $\dbar f=0$ on $\reg$ and $\int_X\frac{1}{\varepsilon}|f|^2_{h_E\otimes h,\omega}\,\dvo<+\infty$, 
    there exists $u\in L^2_{n,q-1}(X,E\otimes L;\omega,h_E\otimes h)$ satisfying $\dbar u=f$ on $\reg$ and 
    \begin{align*}
        \int_X|u|^2_{h_E\otimes h,\omega}\,\dvo\leq\frac{1}{q}\int_X\frac{1}{\varepsilon}|f|^2_{h_E\otimes h,\omega}\,\dvo.
    \end{align*}
    
    Furthermore, if $h_E$ is Nakano semi-positive on $\reg$, then the global $L^2$-estimate holds for every $q>0$. 
    Under the same assumption, even if $X$ is not \kah and $\omega$ is simply a Hermitian metric, the global $L^2$-estimate still holds only in the case $q=1$.
\end{theorem}

The curvature condition on $h$ is required only on $\reg$, and is therefore weaker than singular positivity. 
Furthermore, if $X$ is normal, or if $X$ is locally irreducible and any weight function is locally bounded from above in a neighborhood of $\xs$, then these two conditions are equivalent (see Proposition \ref{Proposition: curvature condition of sHm}).

\begin{proof}
    Fix a smooth Hermitian metric $h_0$ on $L$. Then there exists a globally defined locally integrable weight function $\varphi\in L^1_{loc}(X)$ such that $h=h_0e^{-\varphi}$.
    By the curvature assumption on $h$, the function $\varphi$ is quasi-plurisubharmonic on $\reg$. 
    Fix $q>0$ satisfying $m\geq\min\{n-q+1,\rom{rank}\,E\}$. 
    Since
    \begin{align*}
        \iO{E\otimes L,h_E\otimes h}=\iO{E\otimes L,h_E\otimes h_0}+\idd\varphi\geq\iO{E,h_E}\oid{L}+\varepsilon\omega\oid{E\otimes L}\geq_{n-q+1}\varepsilon\omega\oid{E\otimes L}
    \end{align*}
    holds on $\reg$ in the sense of currents, it follows that 
    \begin{align*}
        \iO{E\otimes L,h_E\otimes h_0}+(\idd\varphi)_{ac}\geq_{n-q+1}\varepsilon\omega\oid{E\otimes L}
    \end{align*}
    holds on $\reg$. Furthermore, for any $(n,q)$-forms, we obtain 
    \begin{align*}
        [\iO{E\otimes L,h_E\otimes h_0}+(\idd\varphi)_{ac},\Lambda_\omega]\geq q\varepsilon\,\rom{id}_{E\otimes L}
    \end{align*}
    on $\bigwedge^{n,q}T^*_{\reg}\otimes E\otimes L$ (see \cite[Lemma 2.4]{Wat25a}).

    We choose an increasing sequence of real numbers $\{c_j\}_{j\in\bb{N}}$ diverging to infinity such that $c_1>\inf_X \varPsi$. 
    For every $j\in\bb{N}$, the above inequality yields 
    \begin{align*}
        \int_{X_{c_j}}\lara{[\iO{E\otimes L,h_E\otimes h_0}\!+\!(\idd\varphi)_{ac},\Lambda_\omega]^{-1}f}{f}_{h_E\otimes h_0,\omega}e^{-\varphi}dV_\omega
        &\leq\frac{1}{q}\int_{X_{c_j}}\frac{1}{\varepsilon}|f|^2_{h_E\otimes h_0,\omega}e^{-\varphi}dV_\omega\\
        &\leq\frac{1}{q}\int_{X}\frac{1}{\varepsilon}|f|^2_{h_E\otimes h,\omega}dV_\omega.
    \end{align*}
    For any $c>\inf_X\varPsi$, the regular locus $(X_c)_{reg}$ admits a complete \kah metric (see \cite[Th\'{e}or\`{e}m 1.5]{Dem82}). 
    Therefore, by applying Theorem \ref{Theorem: [Theorem 5.1, Dem82]}, there exists $u_j\in L^2_{n,q-1}(X_{c_j},$ $E\otimes L;\omega,h_E\otimes h_0e^{-\varphi})$ such that $\dbar u_j=f$ on $(X_{c_j})_{reg}$ and satisfying the $L^2$-estimate 
    \begin{align*}
        \int_{X_{c_j}}|u_j|^2_{h_E\otimes h,\omega}dV_\omega&=\int_{X_{c_j}}|u_j|^2_{h_E\otimes h_0,\omega}e^{-\varphi}dV_\omega\\
        &\leq\int_{X_{c_j}}\lara{[\iO{E\otimes L,h_E\otimes h_0}\!+\!(\idd\varphi)_{ac},\Lambda_\omega]^{-1}f}{f}_{h_E\otimes h_0,\omega}e^{-\varphi}dV_\omega\\
        &\leq\frac{1}{q}\int_{X}\frac{1}{\varepsilon}|f|^2_{h_E\otimes h,\omega}dV_\omega.
    \end{align*}
    Hence, by applying \cite[Lemma 3.18]{Wat25a} for the index $j$, we obtain the desired solution $u\in L^2_{n,q-1}(X,E\otimes L;\omega,h_E\otimes h)$ as the weak limit of a convergent subsequence, satisfying the required $\dbar$-equation $\dbar u=f$ on $\reg$ and the global $L^2$-estimate.

    Finally, we prove that the same statement holds for $q=1$ when $h_E$ is Nakano semi-positive on $\reg$, $X$ is not necessarily \kah and $\omega$ is simply a Hermitian metric. 
    First, as in the above argument, we obtain $\iO{E\otimes L,h_E\otimes h_0}+(\idd\varphi)_{ac}\geq_{1}\varepsilon\omega\oid{E\otimes L}>0$ on $\reg$ and $[\iO{E\otimes L,h_E\otimes h_0}+(\idd\varphi)_{ac},\Lambda_\omega]\geq q\varepsilon\,\rom{id}_{E\otimes L}$ on $\bigwedge^{n,q}T^*_{\reg}\otimes E\otimes L$ for any $q>0$. 
    We choose a complete \kah metric $\widehat{\omega}_j$ on $(X_{c_j})_{reg}$ such that $\omega<\widehat{\omega}_j$ (see \cite[Th\'{e}or\`{e}m 1.5]{Dem82}). 
    By applying \cite[Chapter VIII, Lemma 6.3]{Dem-book}, for any $f\in L^{2,loc}_{n,q}(X,E\otimes L,h_E\otimes h)$, we obtain the integral inequality  
    \begin{align*}
        \int_{X_{c_j}}&\lara{[\iO{E\otimes L,h_E\otimes h_0}\!+\!(\idd\varphi)_{ac},\Lambda_{\widehat{\omega}_j}]^{-1}f}{f}_{h_E\otimes h_0,\widehat{\omega}_j}e^{-\varphi}dV_{\widehat{\omega}_j}\\
        &\hspace{16mm}\leq\int_{X_{c_j}}\lara{[\iO{E\otimes L,h_E\otimes h_0}\!+\!(\idd\varphi)_{ac},\Lambda_\omega]^{-1}f}{f}_{h_E\otimes h_0,\omega}e^{-\varphi}dV_\omega.
    \end{align*}
    
    By applying Theorem \ref{Theorem: [Theorem 5.1, Dem82]} with respect to this complete \kah metric $\widehat{\omega}_j$, there exists $u_j\in L^2_{n,q-1}(X_{c_j},E\otimes L;\widehat{\omega}_j,h_E\otimes h)$ such that $\dbar u_j=f$ on $(X_{c_j})_{reg}$ and 
    \begin{align*}
        \int_{X_{c_j}}|u_j|^2_{h_E\otimes h,\widehat{\omega}_j}dV_{\widehat{\omega}_j}
        &\leq\int_{X_{c_j}}\lara{[\iO{E\otimes L,h_E\otimes h_0}\!+\!(\idd\varphi)_{ac},\Lambda_{\widehat{\omega}_j}]^{-1}f}{f}_{h_E\otimes h_0,\widehat{\omega}_j}e^{-\varphi}dV_{\widehat{\omega}_j}.
    \end{align*}
    In particular, when $q=1$, by applying Lemma \ref{Lemma: inequality of (n,0)-forms} below, we obtain 
    \begin{align*}
        u_j\in L^2_{n,0}(X_{c_j},E\otimes L;\omega,h_E\otimes h)=L^2_{n,0}(X_{c_j},E\otimes L;\widehat{\omega}_j,h_E\otimes h),
    \end{align*}
    together with the integral inequality 
    \begin{align*}
        \int_{X_{c_j}}|u_j|^2_{h_E\otimes h,\omega}dV_{\omega}&=\int_{X_{c_j}}|u_j|^2_{h_E\otimes h,\widehat{\omega}_j}dV_{\widehat{\omega}_j}\\
        &\leq\int_{X_{c_j}}\lara{[\iO{E\otimes L,h_E\otimes h_0}\!+\!(\idd\varphi)_{ac},\Lambda_{\widehat{\omega}_j}]^{-1}f}{f}_{h_E\otimes h_0,\widehat{\omega}_j}e^{-\varphi}dV_{\widehat{\omega}_j}\\
        &\leq\int_{X_{c_j}}\lara{[\iO{E\otimes L,h_E\otimes h_0}\!+\!(\idd\varphi)_{ac},\Lambda_\omega]^{-1}f}{f}_{h_E\otimes h_0,\omega}e^{-\varphi}dV_\omega\\
        &\leq\frac{1}{q}\int_{X}\frac{1}{\varepsilon}|f|^2_{h_E\otimes h,\omega}dV_\omega.
    \end{align*}
    Hence, by applying \cite[Lemma 3.18]{Wat25a} for the index $j$, we obtain the desired solution $u\in L^2_{n,0}(X,E\otimes L;\omega,h_E\otimes h)$ as the weak limit of a convergent subsequence, satisfying the required $\dbar$-equation $\dbar u=f$ on $\reg$ and the global $L^2$-estimate. 
\end{proof}

\begin{lemma}\label{Lemma: inequality of (n,0)-forms}
    Let $\gamma_1$ and $\gamma_2$ be Hermitian metrics on $X$ with $\gamma_1\geq\gamma_2$. 
    Then we have 
    \begin{itemize}
        \item the equality $|u|^2_{\gamma_1} dV_{\gamma_1}=|u|^2_{\gamma_2} dV_{\gamma_2}$ holds for any $(n,0)$-form $u$, 
        \item the inequality $|u|^2_{\gamma_1} dV_{\gamma_1}\leq|u|^2_{\gamma_2} dV_{\gamma_2}$ holds for any $(n,q)$-form $u$ and any $q\geq1$.
    \end{itemize}
\end{lemma}

\begin{remark}\label{Remark: complete Kahler}
    If a weakly pseudoconvex complex space $X$ carries a \kah metric, then for any $c>\inf_X\varPsi$, the regular locus $(X_c)_{reg}$ admits a complete \kah metric \cite[Th\'{e}or\`{e}m 1.5]{Dem82}. 
    In particular, even after removing an additional analytic subset $A$, the space $(X_c)_{reg}\setminus A$ still carries a complete \kah metric. 
    However, in general, without relative compactness it is unclear whether $\reg$ admits a complete \kah metric. 
    
    On the other hand, when $X$ has no singularities, if the analytic subset $Z$ to be removed is compact, or if $Z$ is the zero locus $Z=\sigma^{-1}(0)$ of a holomorphic section $\sigma$ of a vector bundle admitting a smooth Hermitian metric whose Griffiths positivity is bounded above, 
    then $X\setminus Z$ admits a complete \kah metric globally $($see \cite[Lemma 11.9]{Dem12}$)$. 
\end{remark}

If $X$ is compact or a $1$-convex space, in particular Stein, it is known that $\reg$ and $\reg\setminus A$ admit complete \kah metrics (see \cite{Dem82}).

\begin{conjecture}\label{Conjecture: no existence of complete Kahler}
    Let $X$ be a weakly pseudoconvex \kah complex space. Does there exist an example for which $\reg$ does not admit a complete \kah metric? 
    Furthermore, when $X$ is a weakly pseudoconvex \kah manifold, does there exist an analytic subset $A$ such that $X\setminus A$ does not admit a complete \kah metric?
\end{conjecture}

\section{The weak $\overline{\partial}$-operator and its $L^2$-Dolbeault complex}\label{Section: dbar-operator and L2-Dolbeault complexes}

In this section, we provide proofs of Theorems \ref{Theorem: fine Dolbeault resolution} and \ref{Theorem: L2-Dolbeault isomorphism}.
Let $X$ be a complex space of pure dimension $n$ with a Hermitian metric $\omega$, $L\longrightarrow X$ be a holomorphic line bundle with a singular Hermitian metric $h$ and $U\subset X$ be an open subset. 
Recall that 
\begin{align*}
    L^{2,loc}_{p,q}(U,L,h):=\Bigl\{u\in L^{2,loc}_{p,q}(\ureg,L,h)\,\Big|\, u|_{K_{reg}}\in L^2_{p,q}(K_{reg},L,h),\forall\,K\Subset U\Bigr\}.
\end{align*}  
It is easy to check that the presheaves given as $\scr{L}^{2,loc}_{p,q}(L,h)(U):=L^{2,loc}_{p,q}(U,L,h)$ are already sheaves $\scr{L}^{2,loc}_{p,q}(L,h)\longrightarrow X$. 
On $L^{2,loc}_{p,q}(U,L,h)$, we denote by 
\begin{align*}
    \dbar(U):L^{2,loc}_{p,q}(U,L,h)\longrightarrow L^{2,loc}_{p,q+1}(U,L,h)
\end{align*}
the $\dbar$-operator in the sense of distributions on $\ureg=U\setminus\xs$ which is closed and densely defined. 
This $\dbar$-operator is often denoted by $\dbar_w$ or $\dbar_{max}$.
Since $\dbar$ is a local operator, 
we can define the presheaves of germs of forms in the domain of $\dbar$, 
\begin{align*}
    \scr{L}^{p,q}_{L,h}:=\scr{L}^{2,loc}_{p,q}(L,h)\cap \dbar^{-1}\scr{L}^{2,loc}_{p,q+1}(L,h),
\end{align*}
given by $\scr{L}^{p,q}_{L,h}(U)=\scr{L}^{2,loc}_{p,q}(L,h)(U)\cap\rom{Dom}\,\dbar(U)$. 
This sheaf $\scr{L}^{p,q}_{L,h}$ is well known when $X$ is smooth (see \cite{Dem12,Wat25a}), and admits various types of $L^2$-Dolbeault resolutions (see \cite{Rup14,SZ23,SZ24,Wat25a,Wat26}). 
Furthermore, it is easy to see that the sheaves $\scr{L}^{p,q}_{L,h}$ admit partitions of unity, then we obtain sequences of fine sheaves 
\begin{align*}
    \scr{L}^{p,0}_{L,h}\overset{\dbar}{\longrightarrow}\scr{L}^{p,1}_{L,h}\overset{\dbar}{\longrightarrow}\scr{L}^{p,2}_{L,h}\overset{\dbar}{\longrightarrow}\cdots.
\end{align*}
For a holomorphic vector bundle $E\longrightarrow X$ with a smooth Hermitian metric $h_E$, we define the sheaf $\scr{L}^{p,q}_{E\otimes L,h_E\otimes h}$ in the same way.

Let $\pi:\tx\longrightarrow X$ be a canonical desingularization. 
Here, the smooth metric $\pi^*\omega$ is positive on $\tx\setminus \exc$ and semi-positive on $\tx$.
Let $\gamma$ be a Hermitian metric on $\tx$, then $\pi^*\omega\lesssim\gamma$ and $\pi^*\omega\sim\gamma$ on compact subsets of $\tx\setminus \exc$.
Since $\gamma$ is positive and $\pi^*\omega$ is semi-positive, there exists a smooth function $g\in \cal{C}^\infty(\tx,\bb{R}_{\geq0})$ such that $dV_{\pi^*\omega}=g\,dV_\gamma$, where $g>0$ on $\tx\setminus \exc$. 
It follows immediately that $g|u|^2_{\pi^*\omega}=|u|^2_\gamma$ for any $(n,0)$-form $u$, and that $|u|^2_\gamma\lesssim_{\tl{U}} g|u|^2_{\pi^*\omega}$ on $\tl{U}\Subset\tx$ for any $(n,q)$-form $u$, hence  
\begin{align*}
    \int_{\tl{U}}|u|^2_\gamma dV_\gamma \lesssim_{\tl{U}} \int_{\tl{U}}g|u|^2_{\pi^*\omega}g^{-1}dV_{\pi^*\omega}=\int_{\tl{U}}|u|^2_{\pi^*\omega}dV_{\pi^*\omega},
\end{align*}
here $1\leq q\leq n$.
Therefore, we obtain 
\begin{align*}
    L^2_{n,q}(\tl{U},\pi^*\omega)\subset L^2_{n,q}(\tl{U},\gamma) \quad \text{and} \quad L^2_{n,q}(\tl{U},\pi^*L;\pi^*\omega,\pi^*h)\subset L^2_{n,q}(\tl{U},\pi^*L;\gamma,\pi^*h),
\end{align*}
for any open subset $\tl{U}\Subset \tx$. 
The pullback of forms under $\pi$ gives the isometry 
\begin{align*}
    \pi^*:L^2_{p,q}(\ureg,L;\omega,h)\longrightarrow L^2_{p,q}(\pi^{-1}(U)\!\setminus\!E,\pi^*L;\pi^*\omega,\pi^*h)=L^2_{p,q}(\pi^{-1}(U),\pi^*L;\pi^*\omega,\pi^*h).
\end{align*}
Hence, together with the above inclusions of $L^2$-spaces, if $U\Subset X$ is a relatively compact open subet, then the following map induced by $\pi^*$ is continuous: 
\begin{align*}
    \pi^*:L^2_{n,q}(\ureg,L;\omega,h)\longrightarrow L^2_{n,q}(\pi^{-1}(U),\pi^*L;\gamma,\pi^*h).  \tag{$\ast$}
\end{align*}

\begin{proposition}\label{Proposition: Ker dbar}
    If any local weight of $h$ is uniformly bounded from above almost everywhere, then we have the isomorphism $\ogr(h)\otimes\cal{O}_X(L)\cong\rom{Ker}\,\dbar_X\subset\scr{L}^{n,0}_{L,h}$.
\end{proposition}

\begin{proof}  
    It follows from $(\ast)$ that $\pi^*$ defines a morphism of $L^2$-Dolbeault complexes 
    \begin{align*}
        \pi^*:(\scr{L}^{n,\ast}_{L,h},\dbar)\longrightarrow (\pi_*(\scr{L}^{n,\ast}_{\pi^*L,\pi^*h}),\pi_*\dbar), 
    \end{align*} 
    as follows.
    Let $U\subset X$ be an open subset, and let $u\in\scr{L}^{n,q}_{L,h}(U)$ and $f\in\scr{L}^{n,q+1}_{L,h}(U)$ be chosen such that $\dbar u=f$.
    By $(\ast)$, we have $\pi^*u\in\scr{L}^{2,loc}_{n,q}(\pi^{-1}(U),\pi^*L,\pi^*h)$ and $\pi^*f\in\scr{L}^{2,loc}_{n,q+1}(\pi^{-1}(U),\pi^*L,\pi^*h)$ with $\dbar\pi^*u=\pi^*f$ on $\pi^{-1}(U)\setminus \exc$. 
    It follows from the $\dbar$-extension lemma (see \cite[Lemma 6.9]{Dem82}) that $\dbar\pi^*u=\pi^*f$ on $\pi^{-1}(U)$. 
    Hence, we obtain $\pi^*u\in\scr{L}^{n,q}_{\pi^*L,\pi^*h}(\pi^{-1}(U))$ and $\pi^*f\in\scr{L}^{n,q+1}_{\pi^*L,\pi^*h}(\pi^{-1}(U))$, and the above map $\pi^*$ is in fact a morphism of $L^2$-Dolbeault complexes. 
    Including $\rom{Ker}\,\dbar_X:=\rom{Ker}\,(\dbar_X:\scr{L}^{n,0}_{L,h}\longrightarrow\scr{L}^{n,1}_{L,h})\subset\scr{L}^{n,0}_{L,h}$ and 
    $\rom{Ker}\,\dbar_{\tx}:=\rom{Ker}\,(\dbar_{\tx}:\scr{L}^{n,0}_{\pi^*L,\pi^*h}\longrightarrow\scr{L}^{n,1}_{\pi^*L,\pi^*h})\subset\scr{L}^{n,0}_{\pi^*L,\pi^*h}$, we obtain the follows commutative diagram:
\[
\small
\begin{CD}
0 @>>> \rom{Ker}\,\dbar_X
  @>>> \scr{L}^{n,0}_{L,h}
  @>{\dbar_X}>> \scr{L}^{n,1}_{L,h}
  @>{\dbar_X}>> \scr{L}^{n,2}_{L,h}
  @>{\dbar_X}>>  \\
@. @VV{\pi^*}V
   @VV{\pi^*}V
   @VV{\pi^*}V
   @VV{\pi^*}V \\
0 @>>> \pi_*(\rom{Ker}\,\dbar_{\tx})
  @>>> \pi_*(\scr{L}^{n,0}_{\pi^*L,\pi^*h})
  @>{\pi_*\dbar_{\tx}}>>
     \pi_*(\scr{L}^{n,1}_{\pi^*L,\pi^*h})
  @>{\pi_*\dbar_{\tx}}>>
     \pi_*(\scr{L}^{n,2}_{\pi^*L,\pi^*h})
  @>{\pi_*\dbar_{\tx}}>> 
\end{CD}
\]

    Here, it is already known that $\rom{Ker}\,\dbar_{\tx}\cong K_{\tx}\otimes\pi^*L\otimes\scr{I}(\pi^*h)$ (see \cite{Dem12}, \cite[Theorem 5.3]{Wat25a}), 
    and together with the functoriality property \cite[Proposition 5.8]{Dem12} this implies $\pi_*(\rom{Ker}\,\dbar_{\tx})=\pi_*(K_{\tx}\otimes \pi^*L\otimes\scr{I}(\pi^*h))=\ogr(h)\otimes\cal{O}_X(L)$; in particular, the desired isomorphism
    \begin{align*}
        \rom{Ker}\,\dbar_X\cong\pi_*(\rom{Ker}\,\dbar_{\tx})=\ogr(h)\otimes\cal{O}_X(L)
    \end{align*}
    is obtained.
    In fact, since the isomorphism $\scr{L}^{2,loc}_{n,0}(L,h)\cong\pi_*(\scr{L}^{2,loc}_{n,0}(\pi^*L,\pi^*h))$ holds and the $\dbar$-equation extends over the exceptional set (see \cite[Lemma 6.9]{Dem82}), the left vertical arrow in this commutative diagram is an isomorphism.
\end{proof}

Furthermore, the same isomorphism $\ogr(h)\otimes\cal{O}_X(E\otimes L)\cong\rom{Ker}\,\dbar\subset\scr{L}^{n,0}_{E\otimes L,h_E\otimes h}$ holds after adding a smooth Hermitian metric $h_E$ on $E$. 
By Proposition \ref{Proposition: Ker dbar}, if any local weight of $h$ is uniformly bounded from above almost everywhere, then $\ogr(h)\otimes\cal{O}_X(L)$ coincides with the well-known multiplier $S$-sheaf $S_X(L,h)$ (see \cite[Definition 4.1]{SZ24}).

\begin{proof}[Proof of Theorem \ref{Theorem: L2-Dolbeault isomorphism}]
    The $L^2$-Dolbeault complex $(\scr{L}^{p,\ast}_{\pi^*E\otimes\pi^*L,\pi^*h_E\otimes\pi^*h},\dbar_{\tx})$ on $\tx$ is a fine resolusion of $\Omega^p_{\tx}\otimes\cal{O}_{\tx}(\pi^*E\otimes\pi^*L)\otimes\scr{I}(\pi^*h)$ for any integer $p\geq0$; that is, the sequence of sheaves
    \begin{align*}
        0\longrightarrow\Omega^p_{\tx}\otimes\cal{O}_{\tx}(\pi^*E\otimes\pi^*L)\otimes\scr{I}(\pi^*h)\longrightarrow\scr{L}^{p,\ast}_{\pi^*E\otimes\pi^*L,\pi^*h_E\otimes\pi^*h}
    \end{align*}
    is exact (see \cite[Theorem 5.3]{Wat25a}). 
    In particular, in the case $p=n$, the $L^2$-Dolbeault complex $(\scr{L}^{n,\ast}_{\pi^*E\otimes\pi^*L,\pi^*h_E\otimes\pi^*h},\dbar_{\tx})$ is also exact; hence, by using the Leray spectral sequence and the vanishing of higher direct images  
    \begin{align*}
        R^q\pi_*(K_{\tx}\otimes\pi^*E\otimes\pi^*L\otimes\scr{I}(\pi^*h))=0
    \end{align*}
    for any $q>0$ (see \cite[Corollary 1.2]{Fuj13}), 
    the lower line of the commutative diagram obtained by after tensoring with $(E,h_E)$ is a fine resolution of $\pi_*(K_{\tx}\otimes\pi^*E\otimes\pi^*L\otimes\scr{I}(\pi^*h))=\ogr(h)\otimes\cal{O}_X(E\otimes L)$. 
    Note that the direct image of a fine sheaf under $\pi_*$ is again a fine sheaf. 
\end{proof}

\begin{remark}\label{Remark: Thms condition of pi add blow ups}
    In Theorem \ref{Theorem: L2-Dolbeault isomorphism}, it suffices for $\pi:\tx\longrightarrow X$ to be a composition of a resolution of singularities and finitely many local blow-ups, and the statements remain valid even if additional blow-ups of certain analytic subsets are composed.
    In particular, in the case where $X$ is smooth, Theorem \ref{Theorem: L2-Dolbeault isomorphism} also holds when $\pi:\tx\longrightarrow X$ is taken as the blow-up of some analytic subset.
\end{remark}

Finally, we show that the upper line of the commutative diagram above is also exact; this is precisely the proof of Theorem \ref{Theorem: fine Dolbeault resolution}.
Let $(M,\omega)$ be Hermitian complex manifold, $F\longrightarrow M$ be a holomorphic vector bundle equipped with a singular Hermitian metric $h_F$. 
Let $H^{p,q}_{L^2}(M,F;\omega,h_F)_{max}$ (resp. $H^{p,q}_{L^2}(M,F;\omega,h_F)_{min}$) be the $L^2$-Dolbeault cohomology on $M$ with respect to the maximal closed extension $\dbar_{max}$ (resp. the minimal closed extension $\dbar_{min}$).  
Here, $\dbar_{max}$ is the $\dbar$-operator in the sense of distributions on $M$, i.e., $\dbar$ or $\dbar_w$, and $\dbar_{min}$ is the the minimal closed Hilbert space extension of the $\dbar$-operator on smooth forms with compact support, i.e., $\dbar_s$.

\begin{proof}[Proof of Theorem \ref{Theorem: fine Dolbeault resolution}]
    By Proposition \ref{Proposition: Ker dbar}, the Dolbeault complex $(\scr{L}^{n,\ast}_{E\otimes L,h_E\otimes h},\dbar)$ is exact in degree $q=0$, and it is also known to be exact in degrees $q\geq1$ when restricted to the regular locus $\reg$ (see \cite[Theorem\,5.3]{Wat25a}). 
    Hence, it suffices to show that it is exact in degrees $q\geq1$ along the singular locus $\xs$. 
    
    For any fixed singular point $x_0\in\xs$, take a relatively compact Stein open neighborhood $U$ of $x_0$ which admits a holomorphic closed embedding $\iota_U:U\hookrightarrow\bb{C}^N$ and trivializations $\tau:L|_U\overset{\simeq}{\longrightarrow} U\times\bb{C}$ and $\rho:E|_U\overset{\simeq}{\longrightarrow} U\times\bb{C}^{\otimes\rom{rank}\,E}$.
    We can choose a smooth Hermitian metric $h_{E_U}$ on $E|_U$ such that it is Nakano semi-positive on $\ureg$ and $h_E\sim h_{E_U}$. In particular, $h_{E_U}$ can be chosen to be Nakano flat.
    Here $(z_1,\ldots,z_N)$ denote the standard coordinates on $\bb{C}^N$, and taking $\psi:=\iota^*_U|z|^2=\iota^*_U\sum^N_{j=1}|z_j|^2$ as a smooth strictly plurisubharmonic function on $U$, the space $U$ carries a \kah metric $\omega:=\idd\psi=\iota_U^*\idd|z|^2$. 
    Let $\varphi$ be a weight function of $h$ with respect to trivialization $\tau$, then $\varphi$ is quasi-plurisubharmonic by the assumption. 
    Since $U$ is relatively compact, there exists a constant $C>0$ such that $\iO{L,he^{-C\psi}}=\iO{L,h}+C\idd\psi=\iO{L,h}+C\omega\geq\omega$ on $U$ in the sense of currents, and we obtain $L^{2,loc}_{p,q}(U,E\otimes L,h_E\otimes h)=L^{2,loc}_{p,q}(U,E\otimes L,h_{E_U}\otimes he^{-C\psi})$; in particular, $\scr{L}^{p,q}_{E\otimes L,h_E\otimes h}(U)=\scr{L}^{p,q}_{E\otimes L,h_{E_U}\otimes he^{-C\psi}}(U)$. 
    By $L^2$-existence Theorem \ref{Theorem: global L2-estimates}, for any $\dbar$-closed $f\in L^2_{n,q}(U,E\otimes L;\omega,h_{E_U}\otimes he^{-C\psi})=L^2_{n,q}(U,E\otimes L;\omega,h_E\otimes h)$, there exists a solution $u\in L^2_{n,q-1}(U,E\otimes L;\omega,h_{E_U}\otimes he^{-C\psi})$ satisfying $\dbar u=f$ on $U_{reg}$ and the $L^2$-estimate $q||u||^2_{h_{E_U}\otimes he^{-C\psi},\omega}\leq||f||^2_{h_{E_U}\otimes he^{-C\psi},\omega}<+\infty$; 
    in particular, $u\in L^2_{n,q-1}(U,E\otimes L;\omega,h_E\otimes h)$. 
    Therefore, we obtain the $L^2$-cohomology vanishing 
    \begin{align*}
        H^{n,q}_{L^2}(\ureg,E\otimes L;\omega,h_E\otimes h)_{max}=0
    \end{align*}
    for any $q\geq1$. 
    This implies exactness in degrees $q\geq1$, and the proof is complete.
\end{proof}

\section{Nadel vanishing theorem on weakly pseudoconvex complex spaces}

In this section, we provide Nadel-type vanishing theorems on weakly pseudoconvex complex spaces, taking care to account for the existence of a \kah metric.

\subsection{Compact case}

Note that in general $L^2_{p,q}(X,E\otimes L;\omega,h_E\otimes h)\subsetneq L^{2,loc}_{p,q}(X,E\otimes L,h_E\otimes h)$, and these spaces coincide when $X$ is compact.
Hence, the following immediately follows from Theorem \ref{Theorem: L2-Dolbeault isomorphism}.

\begin{theorem}\label{Theorem: isom of cohomology in compact spaces}
    Let $X$ be a compact complex space of pure dimension $n$ and $\pi:\tx\longrightarrow X$ be any resolusion of singularities. 
    If any local weight function of $h$ on $X$ is quasi-plurisubharmonic, then the pullback of forms under $\pi$ induces a natural isomorphism 
    \begin{align*}
        \pi^*:H^{n,q}_{L^2}(\reg,E\otimes L;\omega,h_E\otimes h)_{max}\overset{\cong}{\longrightarrow} H^q(\tx,K_{\tx}\otimes\pi^*E\otimes\pi^*L\otimes\scr{I}(\pi^*h)).
    \end{align*}
    For any Hermitian metric $\omega$ on $X$ and any $q\geq0$, where we have 
    \begin{align*}
        H^{n,q}_{L^2}(\reg,E\otimes L;\omega,h_E\otimes h)_{max}=H^q(X,\ogr(h)\otimes\cal{O}_X(E\otimes L)). 
    \end{align*}
\end{theorem}

Here, it is known that, for a Hermitian manifold $(M,\omega)$, the equality $\dbar_{max}=\dbar_{min}$ holds, i.e., the density of $\scr{D}^{\ast,\ast}_M$ in $\rom{Dom}\,\dbar_{max}$ with respect to the graph norm, if $\omega$ is complete (see \cite{AV65}, \cite[Chapter VIII]{Dem-book}), 
and that, if the complex space $X$ is compact, then its regular locus $\reg$ admits a complete Hermitian metric (see \cite{Dem82}). 

\begin{remark}
    The Hermitian metric $\omega$ in Theorem \ref{Theorem: isom of cohomology in compact spaces} is required to be defined on $X$. 
    Even though $\reg$ admits a complete Hermitian metric, it is not clear whether we can choose an appropriate $\omega$ on $X$ such that the left-hand side of the isomorphism induced by $\pi^*$ coincides with $H^{n,q}_{L^2}(\reg,E\otimes L;\omega,h_E\otimes h)_{min}$.
\end{remark}

Under the assumption of (relative) compactness, we prove that the vanishing of higher cohomology can be obtained without the existence of a \kah metric. 

\begin{proof}[Proof of Theorem \ref{Theorem: Nadel vanishing on Moishezon}]
    When $X$ is smooth, instead of taking a resolution of singularities, we consider blow ups $\pi:\tx\longrightarrow X$ along the analytic subspace corresponding to the singularity set $Z$ of an approximating metric $h_{\nu_0}$ obtained from the refined Demailly approximation (see \cite[Theorem 3.2]{Wat24}) preserving the multiplier ideal sheaf and producing algebraic singularities for the singular positive Hermitian metric $h$ on $L$.
    Since the approximation has algebraic singularities, it becomes possible to offset the singularities of the suitably pulled-back metric $\pi^*h_{\nu_0}$. Subsequently, applying the Negativity Lemma (see \cite[Lemma 2.2]{Wat25b}), one obtains a positive line bundle $\cal{L}\longrightarrow\tx$ constructed from $(\pi^*L,\pi^*h_{\nu_0})$ (see \cite[Theorem 3.5]{Wat24}).
    Consequently, $\tx$ is projective and admits a \kah metric. 
    (In particular, this approach is useful in that it yields projectivity without relying on Moishezon's result (= Theorem \ref{Theorem: Moishezon-ness}) and can be applied to relatively compact spaces, as required in Theorem \ref{Theorem: in subsection Nadel vanishing on w.p.c. cpx sp}.)
    
    As in the proof of Theorem \ref{Theorem: characterization of big and singular positive}, applying \cite[Lemma 3.2]{Wat25b}, there exists a quasi-plurisubharmonic function $\psi:\tx\longrightarrow[-\infty,+\infty)$ which is smooth on $\tx\setminus\pi^{-1}(Z)$ and a sufficiently small number $\varepsilon_X>0$ such that the singular Hermitian metric on $\pi^*L$ defined by $\pi^*he^{-\varepsilon\psi}$ is singular positive on $\tx$ for any $0<\varepsilon<\varepsilon_X$.
    By strong openness property (= Theorem \ref{Theorem: strong openness property}), we obtain $\scr{I}(\pi^*h)=\bigcup_{0<\varepsilon<\varepsilon_X}\scr{I}(\pi^*he^{-\varepsilon\psi})$ on $\tx$. 
    Therefore, by the compact-ness of $\tx$ and the strong Noetherian property of coherent sheaves (see \cite[Chapter II, (3.22)]{Dem-book}), there exists $0<\varepsilon_0<\varepsilon_X$ such that $\bigcup_{0<\varepsilon<\varepsilon_X}\scr{I}(\pi^*he^{-\varepsilon\psi})=\scr{I}(\pi^*he^{-\varepsilon_0\psi})$ on $\tx$. 
    Setting $\tl{h}:=\pi^*he^{-\varepsilon_0\psi}$ to be the singular Hermitian metric on $\pi^*L$, then $\tl{h}$ is singular positive on $\tx$ and satisfies $\scr{I}(\pi^*h)=\scr{I}(\tl{h})$ on $\tx$.
    Hence, the proof is completed by applying the Nadel vanishing theorem (see \cite{Nad90}) together with Theorem \ref{Theorem: L2-Dolbeault isomorphism} and Remark \ref{Remark: Thms condition of pi add blow ups}. 

    When $X$ is a complex space, let $\pi:\tx\longrightarrow X$ be a resolusion of singularities, then $\tx$ is projective and $\pi^*L$ is also big by Theorem \ref{Theorem: Moishezon-ness}. 
    By Theorem \ref{Theorem: characterization of big and singular positive} and Lemma \ref{Lemma: pullback of big for singular Hermitian metric}, there exist a singular Hermitian metric $h$ on $L$ and a singular Hermitian metric $\tl{h}$ on $\pi^*L$ 
    such that $\iO{L,h}\geq\gamma$ on $\reg$ in the sense of currents for some Hermitian metric $\gamma$, the metric $\tl{h}$ is singular positive and $\scr{I}(\tl{h})=\scr{I}(\pi^*h)$ on $\tx$.
    Hence, by Theorem \ref{Theorem: L2-Dolbeault isomorphism}, the proof is reduced to the case of compact complex manifolds.

    Furthermore, the same vanishing theorem remains valid after twisting by a holomorphic vector bundle endowed with a smooth Hermitian metric that is Nakano semi-positive or Demailly $m$-semi-positive on $\reg$, 
    by applying the Nakano-Nadel type vanishing theorem (see \cite[Theorem 1.1]{Wat25a}).
\end{proof}

As is clear from the above proof, Nadel vanishing Theorem \ref{Theorem: Nadel vanishing on Moishezon} holds not only for the singular positive Hermitian metric constructed from the bigness of $L$, as in Theorem \ref{Theorem: characterization of big and singular positive}, but also for an arbitrary given singular positive Hermitian metric.
More generally, we obtain the following.


\begin{theorem}\label{Theorem: in subsection Nadel vanishing on w.p.c. cpx sp}
    Let $(X,\varPsi)$ be a weakly pseudoconvex complex space of pure dimension $n$ and $c>\inf_X\varPsi$ be arbitrary. 
    If $h$ is singular positive on $X_{c+\tau}$ for some $\tau>0$, then we have the following vanishing 
    \begin{align*}
        H^q(X_c,\ogr(h)\otimes\cal{O}_X(L))=0
    \end{align*}
    for any $q>0$, without assuming the existence of a \kah metric. 
    
    Furthermore, if a smooth Hermitian metric $h_E$ on a holomorphic vector bundle $E\longrightarrow X$ is Demailly $m$-semi-positive on $(X_c)_{reg}$, then the cohomology vanishing 
    \begin{align*}
        H^q(X_c,\ogr(h)\otimes\cal{O}_X(E\otimes L))=0
    \end{align*}
    also holds for any $q>0$ satisfying $m\geq\min\{n-q+1,\rom{rank}\,E\}$.
\end{theorem}

\begin{proof}
    After taking a resolution of singularities $\pi:\tx\longrightarrow X$, we construct, as above, a singular Hermitian metric $\tl{h}$ on $\pi^*L|_{\tx_{c+\tau/2}}$ with singular positivity such that the associated ideal sheaves are preserved, namely, $\scr{I}(\pi^*h)=\scr{I}(\tl{h})$, by using the strong openness property, which can be applied due to the relative compactness of $\tl{X}_{c+\tau/2}:=\pi^{-1}(X_{c+\tau/2})$. 
    We then further blow up along the singular locus arising from Demailly's approximation and construct a positive line bundle to obtain a \kah metric. 
    For this blow-up, we again construct a singular Hermitian metric with singular positivity preserving the corresponding ideal sheaf in the same way. 
    Therefore, we reduce the problem to Nadel vanishing on a weakly pseudoconvex manifold carrying a \kah metric (see \cite[Theorem (5.11)]{Dem12}), and apply Theorem \ref{Theorem: L2-Dolbeault isomorphism} and Remark \ref{Remark: Thms condition of pi add blow ups}.
\end{proof}

\subsection{Non-compact case}

\begin{theorem}\label{Theorem: m-positive-Nadel vanishing on w.p.c. cpx sp}
    Let $X$ be a weakly pseudoconvex \kah complex space of pure dimension $n$ and $L\longrightarrow X$ be a holomorphic line bundle with a singular Hermitian metric $h$. 
    Let $E\longrightarrow X$ be a holomorphic vector bundle with a smooth Hermitian metric $h_E$.
    If $h$ is singular positive on $X$ and $h_E$ is Demailly $m$-semi-positive on $\reg$, then we obtain 
    \begin{align*}
        H^q(X,\ogr(h)\otimes\cal{O}_X(E\otimes L))=0 
    \end{align*}
    for any $q>0$ satisfying $m\geq\min\{n-q+1,\rom{rank}\,E\}$. 
    Furthermore, if $h_E$ is Nakano semi-positive on $\reg$ and $X$ is not necessarily K\"{a}hler, then we have only the first cohomology vanishing $H^1(X,\ogr(h)\otimes\cal{O}_X(E\otimes L))=0$.
\end{theorem}

\begin{proof}
    From Theorem \ref{Theorem: fine Dolbeault resolution}, we obtain the $L^2$-Dolbeault isomorphism 
    \begin{align*}
        H^q(X,\ogr(h)\otimes\cal{O}_X(E\otimes L))\cong H^q(\Gamma(X,\scr{L}^{n,\ast}_{E\otimes L,h_E\otimes h}))
    \end{align*}
    and we prove that the right-hand side vanishes. 
    Let $\omega$ be a \kah metric on $X$. 
    By Proposition \ref{Proposition: curvature condition of sHm}, there exists a positive smooth function $\varepsilon:X\longrightarrow\bb{R}_{>0}$ such that $\iO{L,h}\geq\varepsilon\omega$ on $X$ in the sense of currents.
    We arbitrarily take a convex increasing function $\chi\in\cal{C}^{\infty}(\bb{R},\bb{R})$. 
    For any $\dbar$-closed $(n,q)$-form $f\in \Gamma(X,\scr{L}^{n,q}_{E\otimes L,h_E\otimes h})$, the integral
    \begin{align*}
        \int_X\frac{1}{\varepsilon}|f|^2_{h_E\otimes h,\omega}e^{-\chi\circ\varPsi}\dvo \tag{$\int$}
    \end{align*}
    become convergent if $\chi$ grows fast enough. By the global $L^2$-estimates Theorem \ref{Theorem: global L2-estimates}, 
    there exists $u\in L^2_{n,q-1}(X,E\otimes L;\omega,h_E\otimes he^{-\chi\circ\varPsi})$ satisfying $\dbar u=f$ on $\reg$ and 
    \begin{align*}
        \int_X|u|^2_{h_E\otimes h,\omega}e^{-\chi\circ\varPsi}\dvo\leq\int_X\frac{1}{\varepsilon}|f|^2_{h_E\otimes h,\omega}e^{-\chi\circ\varPsi}\dvo<+\infty. 
    \end{align*}
    By the smoothness of $\chi\circ\varPsi$, it follows that $u\in\Gamma(X,\scr{L}^{n,q-1}_{E\otimes L,h_E\otimes h})$, and hence the desired cohomology vanishing $H^q(X,\ogr(h)\otimes\cal{O}_X(E\otimes L))=0$ is obtained.
    
    The first cohomology vanishing can similarly be obtained using Theorem \ref{Theorem: global L2-estimates}.
\end{proof}

As a remark, although the $L^2$-existence Theorem \ref{Theorem: global L2-estimates} follows from the singular positivity of $h$ only on $\reg$, the finiteness of the integral ($\int$) requires the singular positivity of $h$ on $X$, i.e., $\varepsilon>0$ on $X$.
Theorem \ref{Theorem: Nadel vanishing on w.p.c. cpx sp} follows immediately from Theorem \ref{Theorem: m-positive-Nadel vanishing on w.p.c. cpx sp}.

\begin{remark}
    Even if we attempt to reduce the problem to the known Nadel vanishing on a weakly pseudoconvex manifold $\tx$ via Theorem \ref{Theorem: L2-Dolbeault isomorphism}, difficulties remain: 
    it is unclear whether the pullback $\pi^*L$ admits a singular positive Hermitian metric on the whole of $\tx$, and it is also unclear whether the required global $L^2$-estimates can be obtained with a single fixed singular Hermitian metric.
\end{remark}

Indeed, when $h$ is singular positive on $X$, there exist a smooth positive function $\varepsilon:X\longrightarrow\bb{R}_{>0}$ and a \kah metric $\omega$ such that $\iO{L,h}\geq\varepsilon\omega$ on $X$ in the sense of currents (see Proposition \ref{Proposition: curvature condition of sHm}). 
However, after pulling back, we only obtain $\iO{\pi^*L,\pi^*h}\geq\pi^*(\varepsilon\omega)$ on $\tx$ in the sense of currents and the positivity degenerates along $\exc$. 
Thus, for a given $f\in L^{2,loc}_{n,q}(\tx,\pi^*L,\pi^*h)$, even after choosing an appropriate convex increasing function $\chi\in\cal{C}^{\infty}(\bb{R},\bb{R})$, 
it is not clear whether the integral $\displaystyle\int_{\tx}\frac{1}{\pi^*\omega}|f|^2_{\pi^*h,\pi^*\omega}e^{-\chi\circ\pi^*\varPsi}dV_{\pi^*\omega}$ is finite. 
Furthermore, although $\pi^*\omega$ is \kah on $\tx\setminus\exc$, it is not \kah on $\tx$, and it is unclear whether $\tx\setminus\exc$ admits a complete \kah metric (see Remark \ref{Remark: complete Kahler} and Conjecture \ref{Conjecture: no existence of complete Kahler}).
Hence, it is unclear whether the global $L^2$-estimates can be applied appropriately, and the existence of global solutions to the $\dbar$-equation is also unclear.

We immediately obtain the following Kodaira vanishing theorem on weakly pseudoconvex complex spaces.
In the compact case, Kodaira vanishing is known to hold for the dualizing sheaf $\omega_X$ on projective semi log canonical varieties (see \cite[Theorem 1.8]{Fuj14}), 
and Corollary \ref{Corollary: Kodaira vanishing on w.p.c} may be viewed as a generalization to certain non-compact settings.

\begin{corollary}\label{Corollary: Kodaira vanishing on w.p.c}
    Let $X$ be a weakly pseudoconvex complex space of pure dimension and $L\longrightarrow X$ be a holomorphic line bundle. 
    If $L$ is positive, then we have the following 
    \begin{align*}
        H^q(X,\ogr\otimes\cal{O}_X(L))=0
    \end{align*}
    for any $q>0$.
\end{corollary}

As a related result, in the (resp. relatively) compact case, there is Serre-type vanishing on each sublevel set (see \cite[Chapter V, Theorem 4.3]{GPR94}, resp. \cite[Theorem N']{Fuj75}).


\vspace*{5mm}
\noindent
{\bf Acknowledgement.} 
The author would like to thank Professor Henri Guenancia for his valuable comments on local potentials of currents.
The author is supported by Grant-in-Aid for Research Activity Start-up $\sharp$24K22837 and Grant-in-Aid for Early-Career Scientists $\sharp$26K16989 from the Japan Society for the Promotion of Science (JSPS).



\end{document}